\documentclass[11pt]{article}
\usepackage{amsmath}
\usepackage{amssymb}
\usepackage{amsthm}
\usepackage{mathrsfs}
\usepackage{latexsym}

\usepackage{amsthm}
\usepackage{graphicx}
\usepackage{setspace}
\usepackage{xy}
\input xy
\xyoption{all}
\theoremstyle{plain}
\newtheorem{theorem}{Theorem}[section]
\newtheorem{corollary}[theorem]{Corollary}
\newtheorem{proposition}[theorem]{Proposition}
\newtheorem{lemma}[theorem]{Lemma}

\theoremstyle{remark}
\newtheorem{remark}[theorem]{Remark}
\theoremstyle{definition}
\newtheorem{definition}[theorem]{Definition}
\newtheorem{construction}[theorem]{Construction}
\newtheorem{setting}[theorem]{Setting}
\newtheorem{condition}[theorem]{Condition}
\newtheorem{example}[theorem]{Example}
\newtheorem{proposition-definition}[theorem]{Proposition-Definition}

\newtheorem{convention}[theorem]{Convention}

\newcommand{\vdim}{\mathrm{vdim}}

\newcommand{\C}{\mathbb{C}}

\newcommand{\pp}{\mathbb{P}}
\newcommand{\N}{\mathbb{N}}
\newcommand{\Z}{\mathbb{Z}}
\newcommand{\Q}{\mathbb{Q}}
\newcommand{\vv}{\mathrm{virt}}

\title{Virtual pull-backs}
\author{Cristina Manolache}
\date{}

\begin{document}

\maketitle
\begin{abstract}
 We propose a generalization of Gysin maps for DM-type morphisms of stacks $F\to G$ that admit a perfect relative obstruction theory $E_{F/G}^{\bullet}$, which we call a ``virtual pull-back''. We prove functoriality properties of virtual pull-backs. As applications, we analyze Gromov-Witten invariants of blow-ups and projective bundles.
\end{abstract}

\tableofcontents
\section{Introduction}
Given a regular embedding of schemes $i:X\to Y$ one can construct a ``well-behaved'' morphism  (i.e. a bivariant class) $i^!:A_*(Y)\to A_*(X)$ by applying the Fulton-MacPherson construction (see \cite{f}, Chapter 6). The construction of this morphism can be generalized to the following setting. If $i:X\to Y$ is a closed embedding of schemes and $C_{X/Y}\subset E$ is a closed embedding of the normal cone of $i$ into a rank-$r$ vector bundle, then these data determine a ``well-behaved'' morphism $i^!:A_*(Y)\to A_{(*-r)}(X)$ (see Example 17.6.4 in \cite{f}).
\\The main result of this paper generalizes Fulton's example to stacks and to a larger class of morphisms. These generalizations allow us to view virtual classes (see \cite{lt}, \cite{bf}) as generalized pull-backs. This approach will allow us to deduce certain relations between Gromov-Witten invariants, which were the initial reasons for doing this work. 
\\It should be said that the idea is not entirely new, although we did not find this approach in the literature. The main inspiration point was the ``functoriality property of the Behrend-Fantechi class'' of Kim, Kresch and Pantev in \cite{K:01}. Also, a similar situation appears in Jun Li's $[M,N]^{\vv}$-construction (see \cite{Li:01}).
\\When this paper was in an advanced state, we have been informed of Hsin-Hong Lai's paper on ``Gromov-Witten invariants of blow-ups along manifolds with convex normal bundle'' (\cite{h}). The ideas there are basically the same, with a slightly different flavor. We hope, however, that our point of view will contribute to a clear understanding of this subject.
\\
\\In the first section we recall the notions of normal cones of Behrend-Fantechi and Kresch and prove that these two notions are canonically isomorphic. This allows us to use Kresch's ``deformation to the normal cone'' in the context of ``intrinsic normal cones''.
\\The main idea of the second section is to replace the normal sheaf $N_{F/G}$ of a morphism of stacks $f:F\to G$ with a ``virtual normal bundle''. The appropriate context for this is given by obstruction theories. Precisely, if $f$ is a DM-type morphism of Artin stacks (see Definition \ref{dm}) which admits a perfect relative obstruction theory $E_{F/G}^{\bullet}$ (see \cite{bf}), then we take the virtual normal bundle to be $h^1/h^0((E_{F/G}^{\vee})^{\bullet})$. Using this, we obtain a well-defined morphism $f^!:A_*(G)\to A_*(F)$, that we call a virtual pull-back. As a byproduct of our construction we obtain a generalized notion of virtual fundamental class which applies to some examples of Artin stacks.
\\In Section 3 we show that the virtual pull-back satisfies the usual compatibility conditions. Moreover, when we deal with stacks possessing virtual classes we prove that subject to a very natural compatibility relation between obstructions (see Definition \ref{compatib}) the construction gives a map that sends the virtual class of $G$ to the virtual class of $F$. The statement may also be seen as a generalization of the functoriality property in \cite{bf} and \cite{K:01}. 
\\As an application we provide the answer to a very natural question. Given a smooth projective variety $X$ and its blow-up $p:\tilde{X}\to X$ along some smooth projective subvariety, we would like to know when do certain Gromov-Witten invariants of $X$ and $\tilde{X}$ agree. More precisely, if we start with a given homology class $\beta\in A_1(X)$ and a collection of cohomology classes $\gamma_i\in A^*(X)$, then we can associate a ``lifted'' homology class in $A_1(\tilde{X})$ (see Definition \ref{lifting} for a precise statement) and cohomology classes $p^*\gamma_i\in A^*(\tilde{X})$. One could expect that the Gromov-Witten invariants associated to these data should be equal. This was first analyzed by Gathmann (\cite{G:01}) where $X$ was a convex space and $Y$ a point and by Hu (\cite{H:00}, \cite{H:02}) where it was treated the blow-up along points, curves and surfaces. Recently, it was shown and by Lai (\cite{h}) that (subject to a minor condition) the expectation is true  for genus zero Gromov-Witten invariants of blow-ups along subvarieties with convex normal bundles. Our idea is to show the equality of rational Gromov-Witten invariants for $X$ convex and then ``pull the relation back'' to an arbitrary $X$ (see Proposition \ref{main}). The statement we get should be compared with Theorem 1.6 in \cite{h}.
\\The second application concerns rational Gromov-Witten invariants of projective bundles $p_X:\pp_X(V)\to X$. These were studied by Ruan and Qin (\cite{ru}) where $X$ was taken to be a projective space and by Elezi (\cite{e1}, \cite{e2}) when $V$ is a split bundle and $X$ is a toric variety. Here, we analyze the map induced by $p_X$ between the corresponding moduli spaces of stable maps to $\pp_X(V)$ and $X$. 
\\We also show that a particular case of Costello's push-forward formula follows as an easy consequence of our formalism.
\paragraph{Notation and conventions.}
We work over a fixed ground field.
\\An Artin stack is an algebraic stack in the sense of \cite{lau} of finite type over the ground field. For simplicity we will call Deligne --Mumford stacks DM stacks. 
\\Unless otherwise specified we will try to respect the following convention: we will usually denote schemes by $X,\ Y,\ Z$, etc, Artin stacks by $F,\ G,\ H$, etc. and Artin stacks for which we know that they are not Deligne-Mumford stacks (such as the moduli space of genus-$g$ curves or vector bundle stacks) by gothic letters $\mathfrak{M}_{g},\ \mathfrak{E},\ \mathfrak{F}$, etc. 
\\By a commutative diagram of stacks we mean a 2-commutative diagram of stacks and by a cartesian diagram of stacks we mean a 2-cartesian diagram of stacks.
\\Chow groups for schemes are defined in the sense of \cite{f}; this definition has been extended to DM stacks (with $\Q$-coefficients) by Vistoli (\cite{v}) and to algebraic stacks (with $\Z$-coefficients) by Kresch (\cite{k}). We will consider Chow groups (of schemes/stacks) with $\mathbb{Q}$-coefficients.
\\For a fixed stack $F$ we denote by $\mathcal{D}_F$ the derived category of quasicoherent $\mathcal{O}_F$ modules. Unless otherwise specified, we will denote the derived functors ${\bf{L}}f^*$, ${\bf{R}}f_*$, ${\bf{R}}Hom$, etc. by $f^*$, $f_*$, $Hom$, etc.
\\For a fixed stack $F$ we denote by $L_F$ its cotangent complex defined in \cite{ol}.
 
\paragraph{Acknowledgments.} I would like to thank Dan Abramovich, Claudio Fontanari, Lothar G\"{o}ttsche, Andrew Kresch, Hsin-Hong Lai and Ravi Vakil for useful discussions.
\\I am grateful to my advisors Ionu\c{t} Ciocan-Fontanine and Barbara Fantechi for helpful discussions and suggestions. Special thanks are due to Ionu\c{t} for many helpful conversations and hospitality during a one month stay at the University of Minnesota. I am mostly grateful to Barbara and the anonymous referee for their huge amount of suggestions which led to a major improvement of the exposure.
\\I learned a lot during the Moduli Spaces program at the Mittag-Leffler Institute and at the University of Minnesota. It is a pleasure to thank both institutions for the wonderful research environment and support. 
\\I was partially supported by an ENIGMA grant.
\section{Preliminaries}
\subsection{Intrinsic normal cones to DM-type morphisms}
\subsubsection{Background}
\paragraph{DM-type morphisms}
\begin{definition}\label{dm}
 A morphism $f:F\to G$ of Artin stacks is called of Deligne-Mumford type (or shortly of DM-type) if for any morphism $V\to G$, with $V$ a scheme, $F\times_GV$ is a Deligne-Mumford stack.
\end{definition}
\begin{remark}Let us consider the following Cartesian diagram
\begin{equation}
 \xymatrix {F^{\prime} \ar[r]^{f^{\prime}}\ar[d] & G^{\prime}\ar[d]\\
             F\ar[r]^f& G.}
\end{equation}
If $f$ is a DM-type morphism, then $f^{\prime}$ is a DM-type morphism.
\end{remark}
\begin{remark}Let $f:F\to G$ be morphism of stacks and let $L_f$ be the relative cotangent complex. Then $f$ is of DM-type if and only if $L_f\in \mathcal{D}^{\leq0}_F$
\end{remark}
\paragraph{Obstruction Theories}

\begin{definition}
Let $E^{\bullet}\in\mathcal{D}^{\leq0}_F$. $E^{\bullet}$ is said to be of perfect amplitude if there exists $n\geq0$ such that $E^{\bullet}$ is locally isomorphic to $[E^{-n}\to...\to E^{0}]$, where $\forall i\in\{-n,...,0\}$, $E^i$ is a locally free sheaf.
\end{definition}

\begin{definition}Let $E^{\bullet}\in\mathcal{D}^{\leq0}_X$. Then a homomorphism $\Phi: E^{\bullet}\to L^{\bullet}$ in $\mathcal{D}_F$ is called an obstruction theory if $h^0(\Phi)$ is an isomorphism and $h^{-1}(\Phi)$ is surjective. If moreover, $E^{\bullet}$ is of perfect amplitude, then $E^{\bullet}$ is called a perfect obstruction theory.
\end{definition}
\begin{convention}Unless otherwise stated by a perfect obstruction theory we will always mean of perfect amplitude  contained in $[-1,0]$.
\end{convention}
\paragraph{Cone stacks}
\begin{definition}
Let $X$ be a scheme and $\mathcal{F}$ be a coherent sheaf on $X$. We call $C(\mathcal{F}):=Spec Sym(\mathcal{F})$ an abelian cone over $X$.
\end{definition}
As described in \cite{bf}, Section 1, every abelian cone $C(\mathcal{F})$ has a section $0:X\to C(\mathcal{F})$ and an $\mathbb{A}^1$-action.
\begin{definition} An $\mathbb{A}^1$-invariant subscheme of $C(\mathcal{F})$ that contains the zero section is called a cone over $X$.
\end{definition}
Similarly, Behrend and Fantechi define in \cite{bf} Section 1, abelian cone stacks and cone stacks. Let us recall the definition.
\begin{definition}
Let $F$ be a stack and let $E^{\bullet}:=[E^0\to E^1]$ be an element in $\mathcal{D}_F$ such that $E^{i}$ is an abelian sheaf for $i=0,1$ and $h^i({E^{\bullet}})$ is coherent for $i=0,1$. We call the stack quotient $[E^1/E^0]$ (in the sense of \cite{bf} Section 2) an abelian cone stack over stack $F$ .
\\A cone stack is a closed substack of an abelian cone stack invariant under the action of $\mathbb{A}^1$ and containing the zero section.
\end{definition}
\begin{convention}\label{stack}From now on, unless otherwise stated, by cones we will mean cone-stacks.
\end{convention}
\begin{example}(i) Let $i:X\to Y$ be a closed embedding of schemes. If $\mathcal{I}$ denotes the ideal sheaf of $X$ in $Y$, then $N_{X/Y}=Spec Sym\ \mathcal{I}/\mathcal{I}^2$ is called the normal sheaf of $X$ in $Y$ and $C_{X/Y}:=Spec\oplus_{k\geq 0} \mathcal{I}^k/\mathcal{I}^{k+1}\hookrightarrow N_{X/Y}$ is called the normal cone of $X$ in $Y$.
\\(ii) If $f:F\to G$ is a local immersion of DM-stacks, then Vistoli defines (see \cite{v}, Definition 1.20) the normal cone to $f$ as described below. Let us consider a commutative diagram
\begin{equation}\label{vi}
 \xymatrix {U \ar[r]^{\tilde{f}}\ar[d] & V\ar[d]\\
             F\ar[r]^f& G}
\end{equation}
with $U$, $V$ schemes, the upper horizontal arrow a closed immersion and the vertical arrows \'{e}tale. Then $C_{F/G}$ is the cone obtained by descent from $C_{U/V}$.
\\Note that $C_{F/G}\hookrightarrow N_{F/G}=SpecSym\ h^{-1}(L_{F/G})$.
\end{example}
\subsubsection{Intrinsic normal cones}
\begin{definition}\label{normsheaf}
 Let $f:F\to G$ be a DM-type morphism and let $L_{F/G}\in$ ob $\mathcal{D}(\mathcal{O}_{F})$ be the cotangent complex. Then we denote the stack $$h^1/h^0(L_{F/G}^{\vee}):=h^1/h^0(\tau_{[0,1]}(L_{F/G}^{\vee}))$$ (see \cite{bf}) by $\mathfrak{N}_{F/G}$ and we call it the \emph{intrinsic normal sheaf}.
 \end{definition}
% \begin{remark}\label{derhom} From the definition of ${\bf{R}}\mathcal{H}om$ one easily obtains that $\tau_{[0,1]}(L_{F/G}^{\vee})$ is canonically isomorphic to $\tau_{[0,1]}((\tau_{[-1,0]}L_{F/G})^{\vee})$. This shows that $h^1/h^0(L_{F/G}^{\vee})$ is canonically isomorphic to $h^1/h^0((\tau_{[-1,0]}L_{F/G})^{\vee})$. \end{remark}
\begin{proposition}(Behrend-Fantechi) Let us consider diagram (\ref{vi}) with $F$ a Deligne-Mumford stack, the upper horizontal arrow a closed immersion, $U\to F$ an \'{e}tale morphism and $V\to G$ a smooth morphism. Then for any $U$ and $V$ as above, there exists a unique cone-stack $\mathfrak{C}_{F/G}\subset\mathfrak{N}_{F/G}$ such that $\mathfrak{C}_{F/G}\times_FU=[C_{U/V}/\tilde{f}^*T_{V/G}]$.
\end{proposition}
\begin{definition}\label{befa}We call $\mathfrak{C}_{F/G}$ the \emph{intrinsic normal cone to $f$}.
\end{definition}
\begin{remark} The notion of intrinsic normal cone to a morphism $F\to G$ from a Deligne-Mumford stack to an Artin stack has been introduced in \cite{bf} under the hypothesis $G$ is smooth. However, this hypothesis is not used in the construction of the cone. This restriction is only needed in \cite{bf} in order to define a virtual class (see \ref{extrasmooth} in the next section).
\end{remark}
\begin{lemma}\label{seqcones}
 Let \begin{equation*}
 \xymatrix {F^{\prime} \ar[r]^{f^{\prime}}\ar[d]_{p} & G^{\prime}\ar[d]^q\\
             F\ar[r]^f & G}
\end{equation*}
be a commutative diagram of Artin stacks with $f$ and $f^{\prime}$ of DM-type. Then, the natural morphism $p^*L_{F/G}\to L_{F^{\prime}/G^{\prime}}$ in \cite{ol} induces morphism of abelian cone stacks $$\beta:\mathfrak{N}_{F^{\prime}/G^{\prime}}\rightarrow p^*\mathfrak{N}_{F/G}.$$
\end{lemma}
\begin{proof}
The morphism $p^*L_{F/G}\to L_{F^{\prime}/G^{\prime}}$ induces a morphism of abelian cone stacks
\begin{equation}\label{ceamuitat}h^1/h^0(L_{F^{\prime}/G^{\prime}}^{\vee})\rightarrow h^1/h^0((p^*L_{F/G})^{\vee}).
\end{equation} 
%We have a canonical isomorphism $\tau_{[0,1]}(p^*(L_{F/G}^{\vee}))\to p^*\tau_{[0,1]}(L_{F/G}^{\vee})$. One easy way to see this is to consider the exact triangle $${L_{F/G}^{\vee}}_{\leq 1}\to (L_{F/G}^{\vee})\to (L_{F/G}^{\vee})_{\geq 2}$$ (see \cite{lipman} page 34) and then notice that $p^*((L_{F/G}^{\vee})_{\geq 2})\simeq (p^*(L_{F/G}^{\vee}))_{\geq 2}$. Let us show that 
%\begin{equation}\label{boundedcpx}\tau_{[0,1]}(p^*(L_{F/G}^{\vee}))\simeq \tau_{[0,1]}((p^*L_{F/G})^{\vee}).\end{equation} 
%Let us consider the distinguished triangle ${L_{F/G}}_{\tau\leq -2}\to L_{F/G}\to \tau_{[-1,0]}L_{F/G}$. Let us denote ${L_{F/G}}_{\tau\leq -2}$ by $C$. From the definition of the derived functors $Hom$ and $p^*$ we see that $p^*(C^{\vee})$ (respectively $(p^*C)^{\vee}$) is exact in degree smaller than 2.
Using simplicial resolutions we can represent $L_{F/G}$ by a complex of vector bundles $[...\stackrel{}{\rightarrow} E^{-2}\stackrel{}{\rightarrow} E^{-1}\stackrel{}{\rightarrow} E^0]$ of amplitude $[-\infty, 0]$. By \cite{bf} we have that \begin{equation*}
h^1/h^0((p^*L_{F/G})^{\vee})\simeq h^1/h^0( (\tau_{[-1,0]} p^*L_{F/G})^{\vee}).\end{equation*}
Let us denote by $Coker$ the Cokernel of the map $E^{-2}\to E^{-1}$. We have that $p^*Coker$ is the Cokernel of the map $p^*E^{-2}\to p^*E^{-1}$. This shows that 
\begin{equation}\label{tor}
h^1/h^0( (\tau_{[-1,0]} p^*L_{F/G})^{\vee})=\frac{C(p^*Coker)}{C(p^*E^0)}\simeq \frac{p^*C(Coker)}{p^*C(E^0)}\simeq p^*h^1/h^0({L_{F/G}}^{\vee}).
\end{equation}
Let us now conclude the proof. The isomorphism in \ref{tor} together with the morphism in \ref{ceamuitat} gives an ismorphism $$h^1/h^0(L_{F^{\prime}/G^{\prime}}^{\vee})\rightarrow p^*h^1/h^0(L_{F/G}^{\vee}).$$
\end{proof}
\begin{remark}\label{etale} In notations as in \ref{seqcones} let $p:F^{\prime}\to F$ be an \'etale morphism, $G^{\prime}=G$ and $q$ the identity morphism. Then, $\beta:\mathfrak{N}_{F^{\prime}/G}\rightarrow p^*\mathfrak{N}_{F/G}$ is an isomorphism. To see this let us consider the distinguished triangle $$p^*L_{F/G}\to L_{F^{\prime}/F}\to L_{F^{\prime}/G}\to p^*L_{F/G}[1].$$ The claim follows from the fact that $L_{F^{\prime}/F}=0$.
  
\end{remark}
\begin{proposition}\label{dmincl}
 Let \begin{equation*}
 \xymatrix {F^{\prime} \ar[r]^{f^{\prime}}\ar[d]_{p} & G^{\prime}\ar[d]^q\\
             F\ar[r]^f & G}
\end{equation*}
be a commutative diagram with $F$, $F^{\prime}$ DM-stacks stacks. Then, the morphism of Lemma \ref{seqcones} induces a morphism of cone stacks $\alpha:\mathfrak{C}_{F^{\prime}/G^{\prime}}\rightarrow p^*\mathfrak{C}_{F/G}$. If the diagram is cartesian, then $\alpha$ is a closed immersion. If moreover, $q$ is flat, then $\alpha$ is an isomorphism.
\end{proposition}
\begin{proof}
This is a generalization of Proposition 7.1 in \cite{bf} where the authors treat the case $G$ and $G^{\prime}$ are smooth. We will prove it in several steps.
\\\textit{Step 0.} If $f$ and $f^{\prime}$ are local embeddings of DM-stacks the claim follows from \cite{v} Section 1.
\\\textit{Step 1.} Given $F^{\prime}\stackrel{p}{\rightarrow} F\stackrel{f}{\rightarrow}G$ morphisms we show that the natural morphism $\mathfrak{N}_{F^{\prime}/G}\to p^*\mathfrak{N}_{F/G}$ induces a morphism $\mathfrak{C}_{F^{\prime}/G}\to p^*\mathfrak{C}_{F/G}$. We can check the statement locally. For this, let $r:U\to F$ be an \'etale affine chart and let $M$ be a scheme such that $j:U\hookrightarrow M$ is a closed embedding. Moreover, we can choose $M$ smooth over $G$ such that the following diagram 
\begin{equation*}
 \xymatrix{U\ar[r]^j\ar[d]_r&M\ar[d]\\
F\ar[r]\ar[r]&G}
\end{equation*}
commutes. In the same way we choose $N$ smooth over $M$ such that $U^{\prime}:=U\times_{F}F^{\prime}\hookrightarrow N$ is a closed embedding. Putting all together we have a commutative diagram
\begin{equation*}
 \xymatrix{&&N\ar[d]^{\pi}\\
&&M\ar[d]\\
U^{\prime}\ar[r]^s\ar[rruu]^i&U\ar[r]\ar[ru]^j&G}
\end{equation*}
with $i$ and $j$ closed embeddings. Using Step 0 for these maps we obtain a morphism $C_{U^{\prime}/N}\to s^*C_{U/M}$. On the other hand we have a morphism $T_{N/G}\to \pi^*T_{M/G}$.  From the commutative diagram
\begin{equation*}
 \xymatrix {i^*T_{N/G} \ar[r]\ar[d]&C_{U^{\prime}/N}\ar[d]\\
             s^*j^*T_{M/G} \ar[r]&s^*C_{U/M}}
\end{equation*} 
we obtain a morphism $[C_{U^{\prime}/N}/i^*T_{N/G}]\to [s^*C_{U/M}/s^*j^*T_{M/G}]$ and therefore the conclusion.
\\\textit{Step 2.} Let us first treat the case in which the given diagram is cartesian. From Lemma \ref{seqcones} is is enough to show that the natural morphism $\beta:\mathfrak{N}_{F^{\prime}/G^{\prime}}\rightarrow p^*\mathfrak{N}_{F/G}$ restricts to $\beta:\mathfrak{C}_{F^{\prime}/G^{\prime}}\rightarrow p^*\mathfrak{C}_{F/G}$. This statement can be checked locally. As before, let $r:U\to F$ be an affine chart, $V$ a scheme such that $i:U\hookrightarrow V$ is a closed embedding and $V$ smooth over $G$. Let us now consider $V^{\prime}=V\times_GG^{\prime}$. Then we have the following cartesian diagram
\begin{equation}\label{step}
 \xymatrix {U^{\prime} \ar[r]\ar[d]_s&V^{\prime}\ar[r]\ar[d] & G^{\prime}\ar[d]^q\\
             U\ar[r]^i&V\ar[r] & G.}
\end{equation}
As $U\to V$ is a closed embedding, so is $U^{\prime}\to V^{\prime}$. Using Step 0, we obtain a closed embedding of cones $\alpha:[C_{U^{\prime}/V^{\prime}}/(T_{V^{\prime}/G^{\prime}})_{|U^{\prime}}]\rightarrow [s^*C_{U/V}/(T_{V^{\prime}/G^{\prime}})_{|U^{\prime}}]$. From the isomorphism $(T_{V^{\prime}/G^{\prime}})_{|U^{\prime}}\simeq s^*i^*T_{V/G}$ we obtain a closed embedding $C_{U^{\prime}/G^{\prime}}\rightarrow s^*C_{U/G}$. 
\\If moreover, $q$ is flat, the proof follows from the corresponding statement in Step 0.
\\\textit{Step 3.} Let us consider $F^{\prime\prime}:=F\times_GG^{\prime}$ with canonical maps $p^{\prime}:F^{\prime}\to F^{\prime\prime}$ and $p^{\prime\prime}:F^{\prime\prime}\to F$ which satisfy $p=p^{\prime\prime}\circ p^{\prime}$. By Step 2, we have a morphism $\mathfrak{C}_{F^{\prime\prime}/G^{\prime}}\rightarrow {p^{\prime\prime}}^*\mathfrak{C}_{F/G}$ and by Step 1 we have a natural morphism $\mathfrak{C}_{F^{\prime}/G^{\prime}}\rightarrow{p^{\prime}}^*\mathfrak{C}_{F^{\prime\prime}/G^{\prime}}$. Composing the two morphisms we obtain a morphism $\mathfrak{C}_{F^{\prime}/G^{\prime}}\rightarrow p^*\mathfrak{C}_{F/G}$.
\end{proof}
\begin{remark}\label{injsurj}
Let us consider diagram \ref{step}. By  \cite{bf} a morphism of cones $C_{U^{\prime}/V^{\prime}}\to s^*C_{U/V}$ is a closed immersion if and only if the induced morphism $N_{U^{\prime}/V^{\prime}}\to s^*N_{U/V}$ is a closed immersion. This shows that $\mathfrak{C}_{U^{\prime}/G^{\prime}}\rightarrow p^*\mathfrak{C}_{U/G}$ is a closed immersion if and only if $\mathfrak{N}_{U^{\prime}/G^{\prime}}\rightarrow p^*\mathfrak{N}_{U/G}$ is a closed immersion. As being a closed immersion can be checked (\'etale) locally we see that $\mathfrak{C}_{F^{\prime}/G^{\prime}}\rightarrow p^*\mathfrak{C}_{F/G}$ is a closed immersion if and only if $\mathfrak{N}_{F^{\prime}/G^{\prime}}\rightarrow p^*\mathfrak{N}_{F/G}$ is a closed immersion.
\end{remark}

In the following we generalize the notion of intrinsic normal cone to a DM-type morphism $f:F\to G$ to the case $F$ is an Artin stack (not necessarily a DM-stack).

\begin{construction}\label{locart} Let us consider the following commutative diagram
\begin{equation}\label{artincones}
 \xymatrix {U\ar[d]_r\ar[rd]\\
             F\ar[r]^f& G}
\end{equation}
with $f$ a DM-type morphism, $U$ a scheme and $r$ a smooth morphism. By Lemma \ref{seqcones} we have a morphism 
\begin{equation}
\gamma:\mathfrak{N}_{U/G}\to r^*\mathfrak{N}_{F/G}.
\end{equation}
Let us consider the restriction $\tilde{\gamma}:\mathfrak{C}_{U/G}\to r^*\mathfrak{N}_{F/G}$. We denote the image of $\tilde{\gamma}$ by $\mathfrak{C}_{F/G}|_U$ and we call it the local normal cone on $U$ of $F$ to $G$ .
\end{construction}

Let us now show that local normal cones glue. For this, we need the following easy lemma.
\begin{lemma}
 Let us consider the following commutative diagram
\begin{equation*}
 \xymatrix {U^{\prime}\ar[d]_{r^{\prime}}\ar[rdd]\\U\ar[d]_r\ar[rd]\\
             F\ar[r]^f& G}
\end{equation*}
with $r$ and $r^{\prime}$ smooth morphisms. Then $\mathfrak{C}_{F/G}|_U\times_UU^{\prime}$ is naturally isomorphic to $\mathfrak{C}_{F/G}|_{U^{\prime}}$.
\end{lemma}
\begin{proof}
The claim easily reduces to showing that the following diagram is commutative
\begin{equation*}
 \xymatrix{\mathfrak{C}_{U^{\prime}/G}\ar[r]\ar[d]&(r\circ r^{\prime})^*\mathfrak{N}_{F/G}\ar[d]
\\{r^{\prime}}^*\mathfrak{C}_{U/G}\ar[r]&(r\circ r^{\prime})^*\mathfrak{N}_{F/G}}
\end{equation*}
and this is obvious from the corresponding diagram between normal sheaves.
\end{proof}
This lemma shows that there exists a unique closed subcone $\mathfrak{C}_{F/G}\hookrightarrow \mathfrak{N}_{F/G}$ such that for every diagram (\ref{artincones}) we have $\mathfrak{C}_{F/G}\times_FU=\mathfrak{C}_{F/G}|_U$.
\begin{definition}\label{coneartinstack}
The cone $C_{F/G}$ is called the intrinsic normal cone of $f$, or when there is no risk of confusion the intrinsic normal cone of $F$ to $G$.
\end{definition}
\begin{remark}
Let us consider diagram (\ref{artincones}) with $F$ a DM-stack and $r:U\to F$ an \'etale morphism.  By restricting the natural isomorphism $\mathfrak{N}_{U/G}\to r^*\mathfrak{N}_{F/G}$ to $\mathfrak{C}_{U/G}$ (see Remark \ref{etale}) we obtain an isomorphism of cones $\mathfrak{C}_{U/G}\to r^*\mathfrak{C}_{F/G}$. This shows that when $F$ is a DM-stack $C_f$ of Definition \ref{coneartinstack} is canonically isomorphic to $C_f$ of Definition \ref{befa}. 
\end{remark}
\begin{remark}
 Let us consider diagram (\ref{artincones}), with $r$ a smooth morphism. By Proposition \ref{dmincl} and Remark \ref{injsurj} we obtain a closed embedding $\mathfrak{C}_{U/F}\to\mathfrak{C}_{U/G}$. This shows that the sequence $0\to\mathfrak{N}_{U/F}\to\mathfrak{N}_{U/G}\to\mathfrak{N}_{F/G}|_U\to 0$ induces the exact sequence of cones (in the sense of Definition 1.12 in \cite{bf})
\begin{equation}\label{artco}
0\to\mathfrak{N}_{U/F}\to\mathfrak{C}_{U/G}\to\mathfrak{C}_{F/G}|_U\to 0.
\end{equation}
\end{remark}
\begin{remark} Alternatively to Construction \ref{locart} one can define $\mathfrak{C}_{F/G}|_U$ by the sequence \ref{artco}.
\end{remark}

\begin{proposition}\label{pullbackofcones}
 Let \begin{equation*}
 \xymatrix {F^{\prime} \ar[r]^{f^{\prime}}\ar[d]_{p} & G^{\prime}\ar[d]^q\\
             F\ar[r]^f & G}
\end{equation*}
be a commutative diagram of Artin stacks with $f$ of DM-type. Then, there is an induced morphism of cone stacks $\alpha:\mathfrak{C}_{F^{\prime}/G^{\prime}}\rightarrow p^*\mathfrak{C}_{F/G}$. If moreover, the diagram is cartesian, then $\alpha$ is a closed immersion. If $q$ is flat, then $\alpha$ is an isomorphism.
\end{proposition}
\begin{proof} Let $U\to F$ a smooth scheme over $F$ and $U^{\prime}:=U\times_FF^{\prime}$. Consider $U^{\prime}=U\times_FF^{\prime}$. From the commutative diagram
\begin{equation*}
 \xymatrix {U^{\prime} \ar[r]\ar[d]_q & F^{\prime}\ar[d]\ar[r] & G^{\prime}\ar[d]^q\\
             U\ar[r]&F\ar[r] & G}
\end{equation*}
and Proposition \ref{dmincl} we obtain a morphism
\begin{equation}\label{artc}
\mathfrak{C}_{U^{\prime}/G^{\prime}}\rightarrow q^*\mathfrak{C}_{U/G}. 
\end{equation}
Using sequence (\ref{artco}) for $\mathfrak{C}_{F/G}|_U$ and $\mathfrak{C}_{F^{\prime}/G^{\prime}}|_{U^{\prime}}$ we see that the morphism (\ref{artc}) induces a morphism $\mathfrak{C}_{F^{\prime}/G^{\prime}}|_{U^{\prime}}\to q^* \mathfrak{C}_{F/G}|_U$ which glues to a morphism $\mathfrak{C}_{F^{\prime}/G^{\prime}}\rightarrow q^*\mathfrak{C}_{F/G}$ from the corresponding morphism between normal sheaves.
\\If the diagram is cartesian and if moreover, $q$ is flat, the proof follows similarly.
\end{proof}
\subsection{Normal cones to DM-type morphisms} 
A key tool is Kresch's notion of normal cone of a morphism, which we now recall.
% In this section, we define a pull-back morphism $A_*(G)\stackrel{f^!_{\mathfrak{E}}}{\rightarrow} A_*(F)$ depending on some vector bundle stack $\mathfrak{E}$, in the same way as in \cite{f} Chapter 6, but we replace the condition \textit{``$f$ is a regular embedding''} by a weaker condition. Precisely,  if $C_{F/G}$ denotes the normal cone stack to $f$ introduced in \cite{K:03}, we require the following  
%\\\\$\star$ \textit{there exists a vector bundle stack $\mathfrak{E}$, and a closed immersion $C_{F/G}\hookrightarrow \mathfrak{E}$.}
Let $f:F\to G$ be a morphism of DM-type. The normal cone of $f$, denoted $C_f$ or $C_{F/G}$ was defined by Kresch in \cite{k}, section 5.1 under the assumption $f$ representable and locally separated; it is a cone stack over $F$. 
In \cite{K:03}, Section 5.1 and in the proof of Proposition 1 in \cite{K:01}, Kresch mentions that the definition of $C_f$ and its abelian hull $N_f$ extends to DM-type morphisms. We spell out the definition.
\begin{lemma}\label{reducetoincl}Let $f:F\to G$ be a DM-type morphism of Artin stacks. Then one can construct a commutative diagram (not unique)
\begin{equation}
 \xymatrix {U \ar[r]^{\tilde{f}}\ar[d]& V\ar[d]\\
             F\ar[r]^f& G}
\end{equation}
where $U$ and $V$ are schemes, the vertical arrows are smooth surjective and the top arrow $U\to V$ is  a closed immersion. %Let $R=U\times_FU$ and $S=V\times_GV$. Then $U$ and $V$ in diagram (\ref{locsepar}) can be taken such that the natural map $R\to S$ is a locally closed immersion.
\end{lemma}
\begin{proof}Let $W$ be a smooth atlas of $G$. As $f$ is a DM-type morphism $F\times_GW$ is a DM-stack. Let $U$ be an affine \'etale atlas of $F\times_GW$. Then, there exists a smooth scheme $M$ such that $U\hookrightarrow M$ is a closed embedding of schemes. Taking $V$ to be $M\times W$, we obtain the following commutative diagram with the vertical arrows smooth morphisms and the natural map $\tilde{f}:U\to V$ a closed immersion
\begin{equation}\label{locsepar}\xymatrix {U \ar[r]^-{\tilde{f}}\ar[d]_{\acute{e}t} & V:=M\times W\ar[d]\\
F\times_GW\ar[d]\ar[r]&W\ar[d]\\
             F\ar[r]^f& G.}
\end{equation}
\end{proof}
\begin{lemma}Let $R:=U\times_{F}U$ and $S:=V\times_{G}V$, where $U$ and $V$ are defined in the proof of the previous Lemma. Then the natural map $R\to S$ is a locally closed immersion.
\end{lemma}
\begin{proof}
We can factor the morphism $R\to S$ as $R\to U\times_GU\to S=V\times_GV$. The last map is a closed immersion. Let us now show that the first map is a locally closed immersion. But this follows easily from Proposition \ref{pullbackofcones} and the fact that the following diagram is Cartesian
\begin{equation}
 \xymatrix {U\times_FU \ar[r]\ar[d]& U\times_GU\ar[d]\\
             F\ar[r]^{\Delta}&F\times_GF.}
\end{equation}
\end{proof}
\begin{proposition} (Kresch) Let us consider the cone $C_{R/S}$. There are natural morphisms making $C_{R/S}\rightrightarrows C_{U/V}$ into a smooth groupoid in the category of schemes.
\end{proposition}
\begin{proof} (Sketch) Let $q_1,\ q_2:S\to V$ be the obvious projections. Then we have natural maps $$R=U\times_F U\to U\times_G U\to U\times_GV\simeq U\times_VS,$$ the last isomorphism depending on $q_i$. These maps induce natural maps $s_1,\ s_2:C_{R/S}\to C_{U/V}$
\begin{equation}
 C_{R/S}\rightrightarrows (C_{U\times_VS/S})\times_{U\times_VS}R\simeq C_{U/V}\times_UR\rightarrow  C_{U/V}.
\end{equation}
In the same manner as in \cite{K:03} Section 5.1 the maps $s_i$ are smooth and determine a groupoid.
 \end{proof}
In a completely analogous manner one can define a groupoid $[N_{R/S}\rightrightarrows N_{U/V}]$, where  $N_{R/S}$, $N_{U/V}$ are the normal sheaves (where the normal sheaf $N_{R/S}$ is the abelian hull of the normal cone $C_{R/S}$ of \cite{v}, Definition 1.20). This groupoid defines a stack that we denote $N_{F/G}$.  
\begin{definition}\label{kresch}
 Let $C_{F/G}$ be the stack associated to the groupoid $[C_{R/S}\rightrightarrows C_{U/V}]$ and $N_{F/G}$ the stack associated to the groupoid $[N_{R/S}\rightrightarrows N_{U/V}]$. We call $C_{F/G}$ the normal cone of $f$ and $N_{F/G}$ the normal sheaf of $f$.
\end{definition}
\begin{theorem}\label{family} (Kresch)
 Let $f:F\to G$ be a DM-type morphism of Artin stacks. One can define a deformation space, i.e. a flat morphism $M^{\circ}_FG\to\pp^1$ with general fibre $G$ and special fibre the normal cone $C_{F/G}$. Moreover, for any cartesian diagram
\begin{equation*}
\xymatrix{F^{\prime}\ar[r]^{f^{\prime}}\ar[d]&G^{\prime}\ar[d]\\
F\ar[r]^f&G}
\end{equation*}
there exists an induced morphism $M^{\circ}_{F^{\prime}}G^{\prime}\to M^{\circ}_FG$ that fits into a cartesian diagram
\begin{equation*}
\xymatrix{C_{F^{\prime}/G^{\prime}}\ar[r]\ar[d]&M^{\circ}_{F^{\prime}}G^{\prime}\ar[d]\\
C_{F/G}\ar[r]&M^{\circ}_FG}
\end{equation*}
\end{theorem}
\begin{proof}
A detailed proof can be found in \cite{k}, proposition 13.52 for locally closed immersions. Let us sketch the construction in the general case. As in the case of cones there are natural morphisms making $M^{\circ}_{R/S}\rightrightarrows M^{\circ}_{U/V}$ into a smooth groupoid. Let us denote by $M^{\circ}_{F/G}$ the stack (in general it is not algebraic) associated to the groupoid $[M^{\circ}_{R/S}\rightrightarrows M^{\circ}_{U/V}]$.
\\Let us consider the diagram in Lemma \ref{reducetoincl}. Taking $V^{\prime}:=V\times_GG^{\prime}$ and $U^{\prime}:=U\times_VV^{\prime}$ we obtain a similar diagram for $f^{\prime}$. This gives a morphism of groupoids $M^{\circ}_{F^{\prime}/G^{\prime}}\to M^{\circ}_{F/G}$ which induces a morphism of cones $C_{F^{\prime}/G^{\prime}}\to C_{F/G}$. The diagram we obtain it can be easily seen to be cartesian.
\end{proof}
\begin{remark}
From Theorem \ref{family} it follows that whenever $G$ is of pure dimension $r$, then $C_{F/G}$ is again of pure dimension $r$.
\end{remark}
\noindent Let us now compare the normal cone defined by Kresch with the intrinsic normal cone. The following Lemma in probably well-known to experts, but as we did not find it in the literature, we give a detailed proof for completeness.
\begin{proposition}\label{inc}
 If $f: F\to G$ is a DM-type morphism, then the cone stack $N_{F/G}$ of Definition \ref{kresch} is canonically isomorphic to the intrinsic normal sheaf $\mathfrak{N}_{F/G}$ of Definition \ref{normsheaf} and $C_{F/G}$ is canonically isomorphic to the intrinsic normal cone $\mathfrak{C}_{F/G}$ of Definition \ref{befa}.
\end{proposition}
\begin{proof} We divide the proof in several cases. In what follows we use the notation ``$=$'' for canonical isomorphisms.
\\
\\\textit{Case} 1. If $f$ is a closed embedding of schemes the statement is trivial.
\\
\\\textit{Case} 2. If $f$ is a local embedding of DM stacks, then $N_{F/G}$ and $C_{F/G}$ are obtained by descent on $F$ (see \cite{v}) and hence it suffices to check the statement locally. This shows that the statement follows by the first case.
\\
\\\textit{Case} 3. Let us show that $\mathfrak{N}_{F/G}=N_{F/G}$ when $F$ is a DM stack, $G$ an Artin stack and $f$ factors as
\begin{equation*}
 \xymatrix{&W\ar[d]\\
F\ar[ru]^i\ar[r]^f&G}
\end{equation*}
with $i$ a local embedding and M a smooth stack. Then $\mathfrak{N}_{F/G}=N_{F/W}/i^*T_{W/G}$. Let us take $p:U\to F$, $V$ \'{e}tale covers of $F$ and $W$ such that $f$ lifts to a closed embedding of schemes $\tilde{f}:U\to V$. Then, it suffices to show we have an isomorphism $$N_{U/V}\times_{N_{F/W}/i^*T_{W/G}}N_{U/V}\simeq N_{U\times_FU/V\times_GV}$$ compatible with the groupoid structure.
For this, we see the first term is isomorphic to $\tilde{f}^*T_{W/G}\times N_{U/V}\times_{N_{F/W}}N_{U/V}$ and using $V\to W$ is \'{e}tale we obtain the first term is isomorphic to $\tilde{f}^*T_{W/G}\times N_{U\times_FU/V}$. On the other hand, we know by the previous case that $N_{U\times_FU/V\times_GV}$ is canonically isomorphic to $\mathfrak{N}_{U\times_FU/V\times_GV}$ for which we know it is isomorphic to $\mathfrak{N}_{U\times_FU/V}\times \tilde{f}^* T_{V/G}$. This shows $\mathfrak{N}_{F/G}=N_{F/G}$.
\\
\\\textit{Case} 4. In general, we show $\mathfrak{N}_{F/G}=N_{F/G}$. The proof is very similar to Case 3, above. Let us consider diagram (\ref{locsepar}) with the diagonal map $g:U\to G$.  As $g$ factors a closed embedding followed by a smooth morphism we have by Case 3 that 
\begin{equation}\label{uptr}
 N_{U/G}=\mathfrak{N}_{U/G}= N_{U/V}/T_{V/G}|_{U}.
\end{equation}
In order to analyze the lower triangle of diagram (\ref{locsepar}), we consider the distinguished triangle of relative cotangent complexes
\begin{equation*}
 p^*L_{F/G}\to L_{U/G}\to L_{U/F}\to p^*L_{F/G}[1].
\end{equation*}
As $p: U\to F$ is smooth it is easy to see that we are in the conditions of Proposition 2.7 in \cite{bf} and thus we get a short exact sequence of intrinsic normal sheaves
\begin{equation}\label{downtr}
0\to \mathfrak{N}_{U/F} \to\mathfrak{N}_{U/G}\to p^*\mathfrak{N}_{F/G}\to 0.
\end{equation}
By (\ref{uptr}) and (\ref{downtr}), in a similar way as before we get \textit{local} isomorphisms 
\begin{equation}\label{hlr}N_{U/V}\times_{\mathfrak{N}_{U/G}}N_{U/V}\simeq N_{U/V\times_GV}.
\end{equation} 
Moreover, the same equations (\ref{uptr}) and (\ref{downtr}) give a smooth morphism of abelian cone stacks $N_{U/V}\to\mathfrak{N}_{F/G}$ and in a completely analogous fashion we get morphisms of abelian cone stacks $N_{U\times_F U/V\times_GV}\to N_{U/V}$. This shows we obtain a morphism of abelian cone stacks $$N_{U\times_F U/V\times_GV}\to N_{U/V}\times_{\mathfrak{N}_{F/G}}N_{U/V}.$$ By equation (\ref{hlr}), this morphism is a local isomorphism and thus we have an isomorphism $N_{U/V}\times_{\mathfrak{N}_{F/G}}N_{U/V}\simeq N_{U\times_F U/V\times_GV}$. Checking the diagram below is commutative  
\begin{equation*}
 \xymatrix{N_{U\times_F U/V\times_GV}\ar@<0.5ex>[r]\ar@<-0.5ex>[r]\ar[d]&N_{U/V}\ar[d]\\
N_{U/V}\times_{\mathfrak{N}_{F/G}}N_{U/V}\ar@<0.5ex>[r]\ar@<-0.5ex>[r]& N_{U/V}}
\end{equation*}
we obtain an isomorphism of groupoids and therefore the conclusion.
\\
\\\textit{Case} 5. By Case 4 above, it is enough to check that $C_{F/G}$ is canonically isomorphic to the relative intrinsic normal cone $\mathfrak{C}_{F/G}$ locally. For this, we look at the groupoid $[C_{U/V\times_GV}\rightrightarrows C_{U/V}]$ with the two maps obtained by replacing $F$ with $U$. It is easy to see that $N_{U/V\times_GV
}$ is isomorphic to $N_{U/V}\times\tilde{f}^*T_{V/G}$. Via this isomorphism, the two maps defining the groupoid are the projection and the natural action of $\tilde{f}^*T_{V/G}$ on $C_{U/V}$. This shows $C_{F/G}$ is locally isomorphic to $[C_{U/V}/\tilde{f}^*T_{V/G}]$ and therefore the claim follows.
\end{proof}

\begin{remark}
 By the above Lemma we are allowed to identify the normal cone to a morphism with the intrinsic normal cone. In particular, the above Lemma shows that Definition \ref{kresch} is independent of the choice of $U$ and $V$ in diagram (\ref{locsepar}). Although normal cones are cone stacks, we will use for simplicity the notation $C_{F/G}$ instead of $\mathfrak{C}_{F/G}$.  
\end{remark}
If $X$ is a scheme, $E$ is a vector bundle on $X$ and $i:X\to E$ is the zero section, then $C_{X/E}$ is naturally isomorphic to $E$. We prove a series of successive generalizations of this result.
\begin{example}
Let $G$ be a DM-stack, $E$ a vector bundle on $G$ and $G\to E$ the zero section. Then $C_{G/E}$ is canonically isomorphic to $E$.
\begin{proof}Let $V$ be an \'{e}tale atlas of $G$ and $E_V$ the pull-back of $E$ to $V$, then we can construct a commutative diagram as above and $C_{G/E}$ is obtained by descent from $C_{V/E_V}\simeq E_V$. This shows that $C_{G/E}$ is canonically isomorphic to $E$.
\end{proof}
\end{example}
\begin{example}Let $F\stackrel{f}\rightarrow G$ be a DM-type morphism and $p:E\to G$ a vector bundle on $G$. Let $i:G\to E$ be the zero section, and let $g:F\to E$ be $g:=i\circ f$. Then $C_{F/E}$ is canonically isomorphic to $C_{F/G}\times_Ff^*E$.
\begin{proof}Let us consider the distinguished triangles corresponding to $g$ and $f$ respectively. The morphism $i$ induces a morphism $i^*L_E\to L_G$ and therefore we obtain the following morphism of distinguished triangles
\begin{equation*}
\xymatrix{f^*i^*L_E\ar[r]\ar[d]&L_F\ar[r]\ar[d]&L_{F/E}\ar[r]\ar[d]& f^*i^*L_E[1]\ar[d]\\ f^*L_{G}\ar[r]& L_F\ar[r]\ar@<1ex>[u]& L_{F/G}\ar[r]& f^*L_{G}[1]}.
\end{equation*}
Using $p$ instead of $i$ we obtain in the same way a morphism $L_{F/G}\to L_{F/E}$ and thus we get a morphism $f^*L_{G/E}\oplus L_{F/G}\to L_{F/E}$. To show it is an isomorphism it suffices to show the statement locally. As we may assume $G$ is an affine scheme, it is easy to see that $i^*L_E=L_G\oplus E^{\vee}$. On the other hand, $L_{G/E}=[E^{\vee}\to0]$, where $E^{\vee}$ stays in degree $-1$ and therefore we reduced the problem to showing the triangle
\begin{equation*}
f^* L_{G}\oplus E^{\vee}\to L_F\to L_{F/G}\oplus [f^*E^{\vee}\to0]
\end{equation*}
is distinguished. But this follows trivially from the definition of the mapping cone. This shows %we have a distinguished triangle
%\begin{equation*}
 %\xymatrix{f^*L_{G/E}\ar@<1ex>[r]&L_{F/E}\ar@<1ex>[r]\ar[l]& L_{F/G}\ar[r]\ar[l]&f^* L_{G/E}[1]}
%\end{equation*}
that $h^1/h^0(L_{F/E}^{\vee})$ is isomorphic to $h^1/h^0(L_{F/G}^{\vee})\times_Fh^1/h^0(f^*L_{G/E}^{\vee})$. We have thus obtained $C_{F/E}$ is isomorphic to $C_{F/G}\times_F{f}^*E$.
\end{proof}
\end{example}
\begin{example}\label{vb}Let $F\stackrel{f}\rightarrow G$ be DM-type morphism, $\mathfrak{E}:=E^1/E^0$ a vector bundle stack on $G$. Let $G\stackrel{i}{\rightarrow}\mathfrak{E}$ denote the zero section. If $g:F\to G $ is the composition $F\stackrel{f}{\rightarrow}G\stackrel{i}{\rightarrow}\mathfrak{E}$, then $C_{F/G}$ is naturally isomorphic to $C_{F/G}\times_Ff^*\mathfrak{E}$
\begin{proof}Using the above factorization of the morphism $F\to G$, we see that $C_{F/\mathfrak{E}}=[C_{F/E^1}/E^0]$. Using the previous example for $C_{F/E^1}$, we obtain that the normal cone of $F$ in $\mathfrak{E}$ is isomorphic to $C_{F/G}\times_Ff^*\mathfrak{E}$.
\end{proof}
\end{example}
We include two examples in which the normal cone is a vector-bundle stack.
\begin{example} \label{flat} Let $F\to G$ be a smooth morphism of DM-stacks. Then $C_{F/G}$ is isomorphic to $[F/T_{F/G}]$, hence it is a vector bundle stack. 
\end{example}
\begin{example}Let $X\stackrel{f}\rightarrow Y$ be a morphism of smooth schemes. Then, $U$ and $V$ above can be taken to be $X$ and $X\times Y$ as below
\begin{equation*}
\xymatrix{X\ar[r]^-{id\times f}\ar[d]&{X\times Y}\ar[d]^{\pi_2}\\
X\ar[r]&Y} 
\end{equation*}
where $\pi_2$ is the projection on $Y$. It is then easy to see that the normal cone is $[N_{X/X\times Y}/T_{X}]$ that is a vector bundle stack.
\end{example}

\section{Construction}
In the following we will use a result of Kresch.
\begin{proposition}(\cite{K:03}, Proposition 5.3.2)\label{stratif}
Let $F$ admit a stratification by global quotients  (see \cite{K:03}, Definition 3.5.3). Then, for any \emph{vector bundle stack} $\mathfrak{E}$, we have a canonical isomorphism $s^*:A_*(F)\to A_*(\mathfrak{E})$ .
\end{proposition}
\begin{remark}
Every DM-stack admits a stratification by global quotients.
\\If $G$ admits a stratification by global quotients and $F\to G$ is a DM-type morphism then $F$ admits a stratification by global quotients.
\end{remark}
\subsection{Definition of virtual pull-backs}
\begin{condition} We say that a morphism $f:F\to G$ of algebraic stacks and a vector bundle stack $\mathfrak{E}\to F$ satisfy condition ($\star$) if
\begin{enumerate}
\item $f$ is of DM-type,
\item we have fixed a closed embedding $i:C_{f}\hookrightarrow \mathfrak{E}$.
\end{enumerate}
\end{condition}
\begin{convention}
Will say in short that the pair $(f,\mathfrak{E})$ satisfies condition ($\star$).
\end{convention}

\begin{remark}
Let us consider a Cartesian diagram
\begin{equation*}
 \xymatrix {F^{\prime} \ar[r]\ar[d]_{p} & G^{\prime}\ar[d]^q\\
             F\ar[r]^f & G.}
\end{equation*}
If $\mathfrak{E}$ is a vector bundle on $F$ such that $C_{F/G}\hookrightarrow \mathfrak{E}$ is a closed embedding, then $C_{F^{\prime}/G^{\prime}}\hookrightarrow p^*\mathfrak{E}$ is a closed embedding.
\end{remark}
\begin{construction}
Let $F$ be an Artin stack which admits a stratification by global quotient stacks and $\mathfrak{E}$ a vector bundle stack of (virtual) rank $n$ on $F$ such that $(f,\mathfrak{E})$ that satisfies condition $(\star)$ for $f$, we construct a pull-back map $f^{!}_\mathfrak{E}:A_*(G)\to A_{*-n}(F)$ as the composition
\begin{equation*}
 A_*(G)\stackrel{\sigma}\rightarrow A_*(C_{F/G})\stackrel{i_*}{\rightarrow}A_*(\mathfrak{E})\stackrel{s^*}{\rightarrow} A_{*-n}(F).
\end{equation*}
where
\begin{enumerate}
\item $\sigma$ is defined on the level of cycles by $\sigma(\sum n_i[V_i])=\sum n_i [C_{V_i\times_GF/V_i}]$ 
\item $i_*$ is the push-forward via the closed immersion $i$ 
\item $s^*$ is the morphism of Proposition \ref{stratif}.
\end{enumerate}
By Proposition \ref{pullbackofcones} we have a closed embedding of cones $$C_{V_i\times_GF/V_i}\hookrightarrow C_{F/G}.$$ The fact that $\sigma$ is well defined is a consequence of Theorem \ref{family} (see \cite{K:03} Section 3 for local immersions and Section 5 for the general case). 
 \\Going further, for any cartesian diagram
\begin{equation*}
 \xymatrix {F^{\prime} \ar[r]^{f^{\prime}}\ar[d]_{p} & G^{\prime}\ar[d]^q\\
             F\ar[r]^f & G}
\end{equation*}
such that $F^{\prime}$ admits a stratification by global quotient stacks and $\mathfrak{E}\to F$ satisfies condition $(\star)$ for $f$, let $f^{!}_{\mathfrak{E}}:A_*(G^{\prime})\to A_{*-n}(F^{\prime})$ be the composition
\begin{equation*}
 A_*(G^{\prime})\stackrel{\sigma}\rightarrow A_*(C_{F^{\prime}/G^{\prime}})\stackrel{\alpha}{\rightarrow} A_*(C_{F/G}\times_FF^{\prime})\stackrel{i_*}{\rightarrow}A_*(p^*\mathfrak{E})\stackrel{s^*}{\rightarrow}A_{*-n}(F^{\prime})
\end{equation*}
where $\alpha$ is the morphism from Proposition \ref{pullbackofcones}.
\end{construction}
\begin{definition}\label{vp}
In the notation above, we call $f_{\mathfrak{E}}^!: A_*(G)\to A_*(F)$ a \textit{virtual pull-back}. When there is no risk of confusion we will omit the index.
\end{definition}
\begin{remark}
In this remark we do not respect Convention \ref{stack}. If $E$ is a \emph{vector bundle} such that $(f,E)$ satisfies $(\star)$, then the above construction can be applied to any Artin stack $F$.  It is clear that in order to have $E$ a vector bundle $N_f$ must necessarily be a cone.
\\If $f$ is a locally closed embedding, then $N_f$ is a cone. Under this assumption if $E$ is a \emph{vector bundle} such that $(f,E)$ satisfies $(\star)$, then it is not necessary to ask $F$ to admit a stratification by global quotient stacks (see \cite{K:02}, Theorem 2.1.12 (vi)).
\end{remark}

\begin{remark}
Note that in case $X$, $Y$ are schemes such that $X$ is regularly embedded in $Y$, then the normal bundle of $X$ in $Y$ satisfies condition $\star$ and $i^!_{N_{X/Y}}$ is precisely the refined Gysin pull-back of \cite{f}, Section 6.2. We remark that the pull-back depends on the chosen bundle. For example, if $(i,E)$ satisfies condition $(\star)$ we can construct $i^!_{E\oplus E^{\prime}}$, where $E^{\prime}$ is any other vector bundle. These morphisms will be obviously different from each other.  
\end{remark}
\begin{remark}\label{fultflat}
 If $f:X\to Y$ is a smooth morphism of schemes, then by Example \ref{flat} $C_{X/Y}$ is a vector bundle stack and hence we can construct the  associated virtual pull-back $f^!_{C_{X/Y}}:A_*(Y)\to A_*(X)$. This has already been defined in \cite{K:03} and it is proved that the definition agrees with the usual flat pull-back (see e.g. \cite{f}).
\end{remark}
\begin{proposition}If $F\stackrel{f}\rightarrow G$ is a DM-type morphism and there exists a perfect relative obstruction theory $E_{F/G}^{\bullet}$, then condition $(\star)$ is fulfilled.
\\Conversely, if $F\stackrel{f}\rightarrow G$ is a morphism that satisfies condition ($\star$), then there exists a perfect obstruction theory $E^{\bullet}_{F/G}\to L_{F/G}$ such that $\mathfrak{E}=h^1/h^0({E^{\bullet}_{F/G}}^{\vee})$ which is unique up to quasi-isomorphism.
\end{proposition}
\begin{proof}
By Proposition \ref{inc}, the  normal sheaf $N_{F/G}$ is nothing but $\mathfrak{N}_{F/G}$, so the first statement follows from the definitions. Conversely, given $\mathfrak{E}$ a vector bundle stack with a closed embedding $i:C_{F/G}\to\mathfrak{E}$, we obtain an injective morphism between the abelian hulls of $C_{F/G}$ and $\mathfrak{E}$ which means an injective morphism of cones $\mathfrak{N}_{F/G}\to\mathfrak{E}$. On the other hand, giving a vector bundle stack $\mathfrak{E}$ is equivalent to giving a perfect complex $E^{\bullet}_{F/G}$ such that $\mathfrak{E}:=h^1/h^0({E^{\bullet}_{F/G}}^{\vee})$. By \cite{bf}, Theorem 4.5 we have a closed embedding of abelian cone stacks $$\mathfrak{N}_{F/G}\rightarrow h^1/h^0({E^{\bullet}_{F/G}}^{\vee})$$ if and only if ${E^{\bullet}_{F/G}}$ is an obstruction theory.
\end{proof}
\begin{corollary}\label{extrasmooth}
If $F\stackrel{f}\rightarrow G$ is a DM-type morphism such that there exists a perfect relative obstruction theory $E^{\bullet}_{F/G}$ and $G$ is a stack of pure dimension, then $f^!_{E^{\bullet}_{F/G}}([G])$ is a virtual class of $F$ in the sense of \cite{bf}.
\end{corollary}
\subsection{A fundamental example of Obstruction Theory}
The purpose of this section is to explain an example of obstruction theory which will play a fundamental role in the last section of this paper.

\begin{construction}\label{constr}
 Let $f:F\to G$ be a DM-type morphism and let $F$ and $G$ be DM-stacks having relative obstruction theories with respect to some smooth Artin stack $\mathfrak{M}$. Let us denote them by $E_{F/\mathfrak{M}}^{\bullet}$ and $E_{G/\mathfrak{M}}^{\bullet}$ respectively. Given a morphism $\varphi: f^*E_{G/\mathfrak{M}}^{\bullet}\to E_{F/\mathfrak{M}}^{\bullet}$ commuting with $f^*L_{G/\mathfrak{M}}\to L_{F/\mathfrak{M}}$, we construct a relative obstruction theory $E_{F/G}^{\bullet}$.
\\The morphism $f:F\to G$ induces a distinguished triangle of cotangent complexes
\begin{equation*}
 f^*L_{G/\mathfrak{M}}\to L_{F/\mathfrak{M}}\to L_{F/G}\to f^*L_{G/\mathfrak{M}}[1].
\end{equation*}
Similarly, $\varphi$ gives rise to a distinguished triangle 
\begin{equation}\label{obstr}
 f^*E_{G/\mathfrak{M}}^{\bullet}\stackrel{\varphi}{\rightarrow} E_{F/\mathfrak{M}}^{\bullet}\to E_{F/G}^{\bullet}\to f^*E_{G/\mathfrak{M}}[1]
\end{equation}

hence we have a morphism of distinguished triangles that induces the following morphism in cohomology
\begin{equation*}
\scriptsize
\xymatrix@C=0.43cm{h^{-1}(f^*E_{G/\mathfrak{M}}^{\bullet}) \ar[r]\ar[d] & h^{-1}(E_{F/\mathfrak{M}}^{\bullet})\ar[d]\ar[r]\ar[d] & h^{-1}(E_{F/G}^{\bullet})\ar[r]\ar[d] & h^0(f^*E_{G/\mathfrak{M}}^{\bullet}) \ar[r]\ar[d] & h^0(E_{F/\mathfrak{M}}^{\bullet})\ar[d]\ar[r]\ar[d] & h^0(E_{F/G}^{\bullet})\ar[d]\\h^{-1}(f^*L_{G/\mathfrak{M}}^{\bullet}) \ar[r] & h^{-1}(L_{F/\mathfrak{M}}^{\bullet})\ar[r] & h^{-1}(L_{F/G}^{\bullet})\ar[r]& h^0(f^*L_{G/\mathfrak{M}}^{\bullet}) \ar[r] & h^0(L_{F/\mathfrak{M}}^{\bullet})\ar[r] & h^0(L_{F/G}^{\bullet})}
\end{equation*}
We know that the first two vertical arrows are surjective and by the definition of obstruction theories we get by a simple diagram chase that $E_{F/G}^{\bullet}$ is also an obstruction theory.
\end{construction}
\begin{remark}
 Let us note that the morphism of mapping cones $E_{F/G}^{\bullet}\to L_{F/G}$ is not unique in $\mathcal{D}^{\leq0}_F$ and therefore this procedure gives many \emph{different} relative obstruction theories. 
\end{remark}

\begin{remark}\label{smooth}
 If $G$ is smooth over $\mathfrak{M}$ and $E_{G/\mathfrak{M}}^{\bullet}$ is trivial (i.e. $E_{G/\mathfrak{M}}^{\bullet}=L_{G/\mathfrak{M}}^{\bullet}$), then the above diagram shows that $E_{F/G}$ is perfect in $[-1,0]$.
\end{remark}

\begin{example}\label{easyfunct} A special case of this construction is when $F\to G$ is a locally closed immersion and $G$ is taken to be smooth over $\mathfrak{M}$. Taking $h^{-1}(f^*E_{G/\mathfrak{M}}^{\bullet})=0$ we obtain that $h^{-2}(E_{F/G}^{\bullet})=0$. This makes $E_{F/G}^{\bullet}$ into a perfect obstruction theory concentrated in degree $-1$ and consequently $\mathfrak{E}$ into a vector bundle.
\end{example}
Let us now motivate Definition \ref{vp}. For this, let us assume $E_{F/\mathfrak{M}}^{\bullet}$ and $E_{G/\mathfrak{M}}^{\bullet}$ are perfect in $[-1,0]$. Then on $F$ and $G$ we have well defined virtual classes $[F]^{\vv}$ and $[G]^{\vv}$ respectively and we will show in the following that $f^!_{\mathfrak{E}_{F/G}^{\vee}}$ sends the virtual class of $G$ to the virtual class of $F$. As remarked in the previous example, the situation is particularly nice when $G$ is taken to be smooth over $\mathfrak{M}$.
\begin{example} The basic case.
\\In the notation above, let us suppose $G$ is smooth and $F\stackrel{i}{\hookrightarrow}G$ is a closed substack and there exists a morphism $f^*L_G\stackrel{\varphi}{\rightarrow}E_{F}^{\bullet}$. Suppose $\mathfrak{M}$ has pure dimension. Let $E_{F/G}^{\bullet}$ be the cone of $\varphi$. By Construction \ref{constr} it is a perfect obstruction theory for $i$. Then we have 
\\(i) $(C_{F/G},E_{F/G}^{\bullet})$ induces the same virtual class on $F$ as $(C_{F/\mathfrak{M}},E_{F/\mathfrak{M}}^{\bullet})$.
\\(ii) The pull back defined by $E_{F/G}^{\bullet}$ respects the relation $$i^![G]=[F]^{\vv}.$$
\begin{proof}
As $G$ is smooth, the intrinsic normal cone $\mathfrak{C}_{F}$ defined in \cite{bf} is nothing but $[C_{F/G}/i^*T_{G/\mathfrak{M}}]$. Moreover, $i^*L_{G/\mathfrak{M}}^{\bullet}$ can be represented by a complex concentrated in $0$ and $E_{F/G}^{\bullet}$ by a complex concentrated in $-1$. By abuse of notation, we will indicate the corresponding sheaves by $i^*L_{G/\mathfrak{M}}$ and $E_{F/G}$ respectively. Taking the long exact cohomology sequence of the exact triangle (\ref{obstr}), we see that $E_{F/\mathfrak{M}}^{\bullet}$ is quasi isomorphic to $[E_{F/G}\to i^*L_{G/\mathfrak{M}}]$. Therefore the vector bundle stack $\mathfrak{E}_{F/\mathfrak{M}}:=h^1/h^0({({E_{F/\mathfrak{M}}^{\bullet}})}^{\vee})$ is equal to $[(E_{F/G})^{\vee}/i^*T_{G/\mathfrak{M}}]$. Thus we have the diagram with cartesian faces
\begin{equation*}
 \xymatrix @!=0.2pc{F \ar[rrr]\ar[rrd]\ar[dd] &&&C_{F/\mathfrak{M}}\ar[dd]\\
&&C_{F/G}\ar[dd]\ar[ru]\\
F\ar[rrr]\ar[rrd]&&&\mathfrak{E}_{F/\mathfrak{M}}\\
&&E_{F/G}\ar[ru]}
\end{equation*}
In other words, the morphism $A_*(\mathfrak{E}_{F/\mathfrak{M}})\to A_*(F)$ factorizes through $A_*(E_{F/G})$ as follows:
\begin{align*}
 A_*(\mathfrak{E}_{F})&\to A_*(E_{F/G}^{})\to A_*(F)\\
[C_{F/\mathfrak{M}}]&\mapsto [C_{F/G}]\ \mapsto[F]^{\vv}.
\end{align*}

For the second statement, we just have to note that by our definition $i^![G]=s^*([C_{F/G}])$, and by (i) is precisely $[F]^{\vv}$ as defined in \cite{bf}.
\end{proof}
\end{example}

%\subsection{Generalization}
%\paragraph{Restricted virtual pull-backs.}
%As we already remarked, perfect relative obstruction theories are not likely to exist in general. However, if we restrict ourselves to a smaller group, then under fair assumptions we can still define a pullback. Let us make this precise.
%\\ Let $F$, $G$ be DM-stacks as in Corollary \ref{virtual}, that admit perfect obstruction theories to some (smooth) Artin stack $\mathfrak{M}$ and let us denote the image of the virtual pull-back $f^!_{\mathfrak{E}_{G/\mathfrak{M}}}A_*(\mathfrak{M})\to A_*(G)$ by $A_{\mathfrak{M}}(G)$. For any $\alpha\in A_{\mathfrak{M}}(G)$ there exists $\beta\in A_*(\mathfrak{M})$ such that $$\alpha=p^!_{\mathfrak{E}_{G/\mathfrak{M}}}(\beta).$$ Then we define the \textit{restricted virtual pull-back} $f^!:A_{\mathfrak{M}}(G)\to A_*(F)$ by
%\begin{align*}
 %f^!(\alpha)=p^!_{\mathfrak{E}_{F/\mathfrak{M}}}(\beta).
%\end{align*}

\section{Basic properties}
Once we have defined a ``pull-back'', we want to show it has good properties. Due to the geometric properties of the normal cone (\ref{family}), the proofs follow essentially in the same way as the ones in \cite{f}. The fact that our pull-back defines a bivariant class is analogous to Example 17.6.4 in \cite{f}. The only point we need to be careful, is the functoriality property, where we need a compatibility condition between the vector bundle stacks that replace the normal bundles.

\begin{theorem} \label{relations}Consider a fibre diagram of Artin stacks
\begin{equation*}
 \xymatrix {F^{\prime\prime}\ar[r]\ar[d]_{q} & G^{\prime\prime}\ar[d]^p\\
F^{\prime} \ar[r]^{f^{\prime}}\ar[d]_{g} & G^{\prime}\ar[d]^h\\
             F\ar[r]^f & G}
\end{equation*}
and let us assume 
\begin{enumerate}
\item $F^{\prime}$, $F^{\prime\prime}$ admit stratifications by global quotient stacks,
\item ${\mathfrak{E}}$ is a vector bundle stack of rank $d$ such that $(f,{\mathfrak{E}})$ satisfies condition $(\star)$ for $f$.
\item $\mathfrak{E}$ is isomorphic to a global quotient $[E^1/E^0]$, with $E^1$, $E^0$ vector bundles on $F$.
 
\end{enumerate}
(i) (Push-forward) If $p$ is either a projective morphism of Artin stacks or a proper morphism of DM-stacks and $\alpha\in A_k(G^{\prime\prime})$, then $f^!_{\mathfrak{E}}p_*(\alpha)=q_*f^!_{\mathfrak{E}}\alpha$ in $A_{k-d}(F^{\prime})$.\\
(ii)\ (Pull-back) If $p$ is flat of relative dimension $n$ and $\alpha\in A_k(G^{\prime})$, then $f^!_{\mathfrak{E}}p^*(\alpha)=q^*f^!_{\mathfrak{E}}\alpha$ in $A_{k+n-d}(F^{\prime\prime})$\\
(iii)(Compatibility) If $\alpha\in A_k(G^{\prime\prime})$, then $f^{!}_{\mathfrak{E}}\alpha=f^{\prime!}_{g^*{\mathfrak{E}}}\alpha$ in $A_{k-d}(F^{\prime\prime})$.
\end{theorem}
\begin{proof}
(i) \textit{Step 1.} Let us first assume that $f$ is a closed embedding. Let us show that the diagram of groups commutes
\begin{equation}\label{fsimplif}
 \xymatrix {A_*(G^{\prime\prime}) \ar[r]^{\sigma^{\prime\prime}}\ar[d]_{p_*} & A_*(C_{F^{\prime\prime}/G^{\prime\prime}})\ar[d]^{Q_*}\\
             A_*(G^{\prime})\ar[r]^{\sigma^{\prime}} & A_*(C_{F^{\prime}/G^{\prime}})}
\end{equation}
where $Q$ in diagram (\ref{fsimplif}) is the composition of the closed embedding $C_{F^{\prime\prime}/G^{\prime\prime}}\to q^* C_{F^{\prime}/G^{\prime}}$ of Proposition \ref{pullbackofcones} with the projective (respectively proper) map $q^* C_{F^{\prime}/G^{\prime}}\to C_{F^{\prime}/G^{\prime}}$. This follows similarly to Prop 4.2 in \cite{f}. More precisely, let us consider the following factorizations of the morphisms $\sigma^{\prime}$ and $\sigma^{\prime\prime}$
\begin{equation*}
 \xymatrix {A_*(G^{\prime\prime}) \ar[r]^{pr^*}\ar[d]_{p_*} & A_*(G^{\prime\prime}\times\mathbb{A}^1)\ar[d]^{(p\times id)_*}\ar[r]& A_*(C_{F^{\prime\prime}/G^{\prime\prime}})\ar[d]^{Q_*}\\
             A_*(G^{\prime})\ar[r]^{pr^*} & A_*(G^{\prime}\times\mathbb{A}^1)\ar[r]& A_*(C_{F^{\prime}/G^{\prime}})}.
\end{equation*}
The diagram on the left commutes and we are left to show that the diagram on the right commutes. But the diagram on the right is induced by the commutative diagram below
\begin{equation}\label{tricky}
 \xymatrix {A_*(M^{\circ}_{F^{\prime\prime}}G^{\prime\prime}) \ar[r]\ar[d]_{P_*} & A_*(C_{F^{\prime\prime}/G^{\prime\prime}})\ar[d]^{Q_*}\\
            A_*(M^{\circ}_{F^{\prime}}G^{\prime})\ar[r] & A_*(C_{F^{\prime}/G^{\prime}})}
\end{equation}
where the horizontal maps are the ones induced by the natural inclusions of $C_{F^{\prime\prime}/G^{\prime\prime}}$ (and $C_{F^{\prime}/G^{\prime}}$) in $M^{\circ}_{G^{\prime\prime}}F^{\prime\prime}$ (and respectively $M^{\circ}_{F^{\prime}}G^{\prime}$). The commutativity of this diagram shows that diagram \ref{fsimplif} commutes.\\
\textit{Step 2.} Let $C^{\prime}=C_{F^{\prime}/G^{\prime}}\times_{g^*\mathfrak{E}}g^*E^1$ and $C^{\prime\prime}=C_{F^{\prime\prime}/G^{\prime\prime}}\times_{q^*g^*\mathfrak{E}}q^*g^*E^1$. Let $\eta:E^1\to\mathfrak{E}$ be the natural projection, $\eta^{\prime}:g^*E^1\to g^*\mathfrak{E}$ the morphism induced by $\eta$ and similarly $\eta^{\prime\prime}:q^*g^*E^1\to q^*g^*\mathfrak{E}$. We have that $\eta$ is smooth which implies that $C^{\prime}\to C_{F^{\prime}/G^{\prime}}$ and $C^{\prime\prime}\to C_{F^{\prime\prime}/G^{\prime\prime}}$ are smooth. Let $i^{\prime}: C_{F^{\prime}/G^{\prime}}\to g^*\mathfrak{E}$ and $j^{\prime}: C^{\prime}\to g^* E^1$ be the natural inclusions induced by $i:C_{F/G}\hookrightarrow\mathfrak{E}$. By the commutativity of flat pull-backs with projective push-forwards (see \cite{K:03}) we obtain the following commutative diagram
\begin{equation*}
\xymatrix{ A_*(C_{F^{\prime}/G^{\prime}})\ar[r]\ar[d]_{i^{\prime}_*}& A_*(C^{\prime})\ar[d]^{j^{\prime}_*}\\
A_*(g^*\mathfrak{E})\ar[r]^{{\eta^{\prime}}^*}&A_*(g^*E^1).}
\end{equation*}
Let $0^{\prime}:F^{\prime}\to g^*E^1$ be the zero section and let us denote the composition $F^{\prime}\stackrel{0^{\prime}}{\rightarrow} g^*E^1\stackrel{\eta^{\prime}}{\rightarrow}g^*\mathfrak{E}$ by $s^{\prime}$. Note that $0^{\prime}$ is regular and $\eta^{\prime}$ is smooth. By functoriality of Gysin maps (see \cite{K:03}) we have that ${s^{\prime}}^*={0^{\prime}}^!{\eta^{\prime}}^*$. With this we have shown that $f^!_{\mathfrak{E}}:A_*(G^{\prime})\to A_*(F^{\prime})$ is equal to the composition
\begin{equation}\label{vl1}
A_*(G^{\prime})\stackrel{\sigma^{\prime}}{\rightarrow} A_*(C_{F^{\prime}/G^{\prime}})\stackrel{{\eta^{\prime}}^*}{\rightarrow}A_*(C^{\prime})\stackrel{j^{\prime}_*}{\rightarrow}A_* (g^*E^1)\stackrel{{0^{\prime}}^!}{\rightarrow} A_*(F^{\prime}).
\end{equation}
In the same way we obtain that $f^!_{\mathfrak{E}}:A_*(G^{\prime\prime})\to A_*(F^{\prime\prime})$ is equal to the composition
\begin{equation}\label{vl2}
A_*(G^{\prime\prime})\stackrel{\sigma^{\prime\prime}}{\rightarrow} A_*(C_{F^{\prime\prime}/G^{\prime\prime}})\stackrel{{\eta^{\prime\prime}}^*}{\rightarrow}A_*(C^{\prime\prime})\stackrel{j^{\prime\prime}_*}{\rightarrow} A_*(q^*g^*E^1)\stackrel{{0^{\prime\prime}}^!}{\rightarrow} A_*(F^{\prime\prime}).
\end{equation}
Let $Q:C_{F^{\prime\prime}/G^{\prime\prime}}\to C_{F^{\prime}/G^{\prime}}$ defined analogously to $Q$ in Step 1. It can be easily seen from definitions that the following diagram is cartesian
\begin{equation*}
\xymatrix{C^{\prime\prime}\ar[r]\ar[d]_R&C_{F^{\prime\prime}/G^{\prime\prime}}\ar[d]^Q\\
C^{\prime}\ar[r]&C_{F^{\prime}/G^{\prime}}.}
\end{equation*}
We have thus obtained a diagram
\begin{equation}\label{ocolit}
\xymatrix{A_*(G^{\prime\prime})\ar[r]\ar[d]_{p_*}&A_*(C_{F^{\prime\prime}/G^{\prime\prime}})\ar[r]^{{\eta^{\prime\prime}}^*}& A_*(C^{\prime\prime})\ar[d]^{R_*}\\
A_*(G^{\prime})\ar[r]&A_*(C_{F^{\prime}/G^{\prime}})\ar[r]^{{\eta^{\prime}}^*}& A_*(C^{\prime}).}
\end{equation}
In the following we show that the above diagram commutes. Let $\alpha=[V^{\prime}]$ with $V^{\prime}$ reduced, $V=p(V^{\prime})$, $W=V\times_{G^{\prime}}F^{\prime}$ and $W^{\prime}=V^{\prime}\times_VW$. Proposition \ref{pullbackofcones} and the definition of $Q$ imply that the restriction of $Q$ to $C_{W^{\prime}/V^{\prime}}$ factors as
\begin{equation}\label{virtpf0}  C_{W^{\prime}/V^{\prime}}\to C_{W/V}\to C_{F^{\prime\prime}/G^{\prime\prime}}
\end{equation}
and therefore the restriction of $R$ to $C_{W^{\prime}/V^{\prime}}\times_{q^*g^*\mathfrak{E}} q^*g^* E^1$ factors as follows
\begin{equation}\label{virtpf}C_{W^{\prime}/V^{\prime}}\times_{q^*g^*\mathfrak{E}} q^*g^* E^1\to C_{W/V}\times_{g^*\mathfrak{E}}g^*E^1 \to C^{\prime}.
\end{equation} 
By abuse of notation we will denote the first map in (\ref{virtpf}) again by $R$. Without loss of generality we may assume that $V^{\prime}$ is irreducible. Let $d$ be the degree of $p$ restricted to $V^{\prime}$. Let us prove in the following that $$R([C_{W^{\prime}/V^{\prime}}\times_{q^*g^*\mathfrak{E}} q^*g^* E^1])=d[C_{W/V}\times_{g^*\mathfrak{E}}g^*E^1].$$ Let $r_0$ denote the rank of $E^0$. By (\ref{ocolit}), (\ref{virtpf0}) and (\ref{virtpf}) we have a commutative diagram
\begin{equation}\label{ocolit2}
\xymatrix{A_k(V^{\prime})\ar[r]\ar[d]_p&A_k(C_{W^{\prime}/V^{\prime}})\ar[r]\ar[d]_{Q_*}& A_{k+r_0}(C_{W^{\prime}/V^{\prime}}\times_{q^*g^*\mathfrak{E}} q^*g^* E^1)\ar[d]^{R_*}\\
A_k(V)\ar[r]&A_k(C_{V/W})\ar[r]& A_{k+r_0}(C_{W/V}\times_{g^*\mathfrak{E}}g^*E^1 )}
\end{equation}
where by abuse of notation the restrictions of $p$, $Q$ and $R$ are denoted by $p$, $Q$ and $R$ respectively. As the degree is preserved by flat pull-back is suffices to show that $Q_*([C_{W^{\prime}/V^{\prime}}])=d [C_{W/V}]$. This can be checked locally which means that it is enough to check the statement when $W\to V$ factors as $W\to M\to V$ with the first map a closed embedding and the second a smooth morphism. Let us form the cartesian diagram
\begin{equation*}
\xymatrix{W^{\prime}\ar[r]^{s^{\prime}}\ar[d]&M^{\prime}\ar[r]^{t^{\prime}}\ar[d]_r&V^{\prime}\ar[d]\\
W\ar[r]^s&M\ar[r]^t&V}
\end{equation*}
By definition $C_{W/V}=C_{W/M}/s^*T_{M/V}$ and $C_{W^{\prime}/V^{\prime}}=C_{W^{\prime}/M^{\prime}}/{s^{\prime}}^*T_{M^{\prime}/V^{\prime}}$. As $t$ is smooth we have that ${s^{\prime}}^*T_{M^{\prime}/V^{\prime}}={s^{\prime}}^*r^*T_{M/V} $. Let $r:C_{W^{\prime}/M^{\prime}}\to C_{W/M}$ be the map induced by the map in Proposition \ref{pullbackofcones}. By Step 1 we have that  that $r([C_{W^{\prime}/M^{\prime}}])=d[C_{W/M}]$. This shows that $$R_*([C_{W^{\prime}/M^{\prime}}/{s^{\prime}}^*T_{M^{\prime}/V^{\prime}}])=d[C_{W/M}/s^*T_{M/V}].$$ By sequences (\ref{vl1}) and (\ref{vl2}) and the commutativity of diagram (\ref{ocolit}) we obtain that $f^!_{\mathfrak{E}}p_*(\alpha)=q_*f^!_{\mathfrak{E}}\alpha$.
\\(ii) By (i) it is enough to show the statement for $G^{\prime}$ irreducible and $\alpha=G^{\prime}$. Let $s_1:F^{\prime}\to g^*\mathfrak{E}$ and $s_2:F^{\prime\prime}\to q^*g^*\mathfrak{E}$ be the zero sections. Then using the definition of virtual pull-backs we have that $$q^*f^!_{\mathfrak{E}}[G^{\prime}]=q^*s_1^*C_{F^{\prime}/G^{\prime}}.$$ By the flatness of $p$ we obtain that $f^!_{\mathfrak{E}}p^*(G^{\prime})=f^!_{\mathfrak{E}}G^{\prime\prime}$ and using again the definition of virtual pull-backs we obtain $f^!_{\mathfrak{E}}G^{\prime\prime}=s_2^*C_{F^{\prime\prime}/G^{\prime\prime}}$. Using now Proposition \ref{pullbackofcones} we have that $C_{F^{\prime\prime}/G^{\prime\prime}}=q^*C_{F^{\prime}/G^{\prime}}$. We are thus left to show that $q^*s_1^*C_{F^{\prime}/G^{\prime}}=s_2^*r^{*}C_{F^{\prime}/G^{\prime}}$, where $r$ the obvious flat morphism $q^*C_{F^{\prime}/G^{\prime}}\to C_{F^{\prime}/G^{\prime}} $. Noting that $s_2^*=(s_1)_{\mathfrak{E}}^!$ the last statement is true by the corresponding statement for $s_1$.
\\(iii) Is obvious.
\end{proof}
\begin{remark} If $p$ is projective we do not need $\mathfrak{E}$ to be a global quotient. The complication in Step 2 of the proof of Theorem \ref{relations} (i) is due to the fact that push-forwards along proper morphisms  of \emph{Artin stacks} cannot be defined unless the morphism is projective. Concretely, if $q$ is proper but not projective, we do not know how to define $Q_*$ in diagram (\ref{tricky}).
\\On the contrary, if $f$ is a local embedding of DM stacks and $q$ is proper then, the deformation spaces are DM stacks and $P$ is proper; this implies that we can push-forward cycles along $P$ and $Q$.
\end{remark}

%\begin{remark} From Remark \ref{fultflat}, the generalized Gysin pull-back is well-defined for smooth pull-backs. Let us show that the two definitions agree. By (i) above, it is enough to prove the claim for $\alpha=[G]$, for which it follows trivially by construction.  \end{remark}

\begin{theorem}(Commutativity)
Consider a fiber diagram of Artin stacks
\begin{equation*}
\xymatrix{F^{\prime\prime}\ar[d]_q\ar[r]^v&G^{\prime\prime}\ar[d]\ar[r]^u&H\ar[d]^g\\
F^{\prime}\ar[d]_p\ar[r]&G{^\prime}\ar[d]\ar[r]&K\\
F\ar[r]^f&G}
\end{equation*}such that $F^{\prime}$ and $G^{\prime\prime}$ admit stratifications by global quotients. Let us assume $f$ and $g$ are morphisms of DM-type and let ${\mathfrak{E}}$ and ${\mathfrak{F}}$ be vector bundle stacks of rank $d$ and $e$ respectively such that $(f,{\mathfrak{E}})$ and $(g,{\mathfrak{F}})$ satisfy condition $(\star)$. Then for all $\alpha\in A_k(G^{\prime})$,
\begin{equation*}
 g^!_{\mathfrak{F}}f^!_{\mathfrak{E}}(\alpha)= f^!_{\mathfrak{E}}g^!_{\mathfrak{F}}(\alpha)
\end{equation*}
in $A_{k-d-e}(F^{\prime\prime})$.
\end{theorem}
\begin{proof}
 Using Theorem \ref{relations} we may assume $\alpha=[G^{\prime}]$. We see that the pull-back of $g^!f^![G^{\prime}]$ to $q^*p^*\mathfrak{E}\oplus v^*u^*\mathfrak{F}$ is equal to $C_{C_{F^{\prime}/G^{\prime}}\times_{G^{\prime}}G^{\prime\prime}/C_{F^{\prime}/G^{\prime}}}$ and the pull-back of $f^!g^![G^{\prime}]$ to $q^*p^*\mathfrak{E}\oplus v^*u^*\mathfrak{F}$ is equal to $C_{C_{G^{\prime\prime}/G^{\prime}}\times_{G^{\prime}}F^{\prime}/C_{G^{\prime\prime}/G^{\prime}}}$. By Vistoli's rational equivalence (\cite{v} Lemma 3.16, or equivalently \cite{k}) $$[C_{C_{F^{\prime}/G^{\prime}}\times_{G^{\prime}}G^{\prime\prime}/C_{F^{\prime}/G^{\prime}}}]=[C_{C_{G^{\prime\prime}/G^{\prime}}\times_{G^{\prime}}F^{\prime}/C_{G^{\prime\prime}/G^{\prime}}}]$$ in $A_*(C_{F^{\prime}/G^{\prime}}\times_{G^{\prime}}C_{G^{\prime\prime}/G^{\prime}})$. This equivalence pushes forward to $A_*(p^*q^*\mathfrak{E}\oplus v^*u^*\mathfrak{F})$ and therefore the equality in the theorem follows.
\end{proof}
\begin{theorem}\label{bivar}
Let $F$ be an Artin stacks which admits a stratification by global quotients, let $f:F\to G$ be a morphism and $\mathfrak{E}\to F$ be a rank-$n$ vector bundle stack on $F$ such that $(f,\mathfrak{E})$ satisfies  Condition ($\star$). Then $f_{\mathfrak{E}}^!$ defines a bivariant class in $A^{n}(F\to G)$ in the sense of \cite{f}, Definition 17.1.
\end{theorem}

\begin{definition}\label{compatib}
 Let $F\stackrel{f}\rightarrow G\stackrel{g}\rightarrow \mathfrak{M}$ be DM-type morphisms of stacks. If we are given a distinguished triangle of relative obstruction theories which are perfect in $[-1,0]$ $$g^*E^{\bullet}_{G/\mathfrak{M}}\stackrel{\varphi}\rightarrow E^{\bullet}_{F/\mathfrak{M}}\to E^{\bullet}_{F/G}\to g^*E^{\bullet}_{G/\mathfrak{M}}[1]$$ with a morphism to the distinguished triangle $$g^*L_{G/\mathfrak{M}}\to L_{F/\mathfrak{M}}\to L_{F/G}\to g^*L_{G/\mathfrak{M}}[1],$$ then we call $(E^{\bullet}_{F/G},E^{\bullet}_{G/\mathfrak{M}},E^{\bullet}_{F/\mathfrak{M}})$ a compatible triple.
\end{definition}
\begin{remark}
 As in Construction \ref{constr}, if there is a morphism $E^{\bullet}_{F/G}\stackrel{\psi}\rightarrow g^*E^{\bullet}_{G/\mathfrak{M}}[1]$ compatible with the corresponding morphism between the cotangent complexes, then $\psi$ determines a complex $E^{\bullet}_{F/\mathfrak{M}}$ which fits in a distinguished triangle as above. Moreover, $E^{\bullet}_{F/\mathfrak{M}}$ defines a relative obstruction theory. If $E^{\bullet}_{F/G}$ and $E^{\bullet}_{G/\mathfrak{M}}$ are perfect, then $E^{\bullet}_{F/\mathfrak{M}}$ is perfect.
\end{remark}
\begin{lemma}
Consider a fibre diagram
\begin{equation*}
 \xymatrix {F^{\prime} \ar[r]^{f^{\prime}}\ar[d]_{p} & G^{\prime}\ar[d]^q\ar[r]^{0}&\mathfrak{F}^{\prime}\ar[d]^r\\
             F\ar[r]^f & G\ar[r]^0&\mathfrak{F}.}
\end{equation*}
with $F^{\prime}$ an Artin stack which admits a stratification by global quotients, $\pi:\mathfrak{F}\to G$ a vector bundle stack of rank $e$ on $G$ and $\pi^{\prime}:\mathfrak{F}^{\prime}\to G^{\prime}$ its pullback to $G^{\prime}$. Let us assume $\mathfrak{E}^{\prime}$ is a vector bundle stack on $F$ such that $(f,\mathfrak{E}^{\prime})$ satisfies condition $(\star)$. Then we have a natural map $$C_{F/\mathfrak{F}}\to\mathfrak{E}:=\mathfrak{E}^{\prime}\oplus f^*\mathfrak{F}$$ which is a closed embedding and for any $\alpha\in A_k(\mathfrak{F}^{\prime})$ $$(0\circ f)_{\mathfrak{E}}^!(\alpha)=f^!_{\mathfrak{E}^{\prime}}(0^!_{\mathfrak{F}}(\alpha)).$$ 
\end{lemma}
\begin{proof}
 For the first part it suffices to show that $C_{F/\mathfrak
{F}}$ is canonically isomorphic to $C_{F/G}\times_{F}(C_{G/\mathfrak{F}}\times_{G}F)$, that is example \ref{vb}.
\\The equality follows in the same way as in (\cite{f}). Let us notice that by theorem \ref{relations} (i) and the fact that $(\pi^{\prime})^*:A_*({F}^{\prime})\to A_{*+e}(\mathfrak{F}^{\prime})$ is an isomorphism (Theorem 2.1.12, part (x) in \cite{K:02}) we may assume $\alpha$ to be represented by $\mathfrak{F}^{\prime}$ and $G^{\prime}$ can be taken to be irreducible. Now, the problem reduces to 
\begin{equation}\label{fult}
 (0\circ f)^![\mathfrak{F}^{\prime}]=f^![G^{\prime}].
\end{equation}
If $\pi_1:p^*\mathfrak{E}\to p^*\mathfrak{E}^{\prime}$ and $\pi_2:p^*\mathfrak{E}^{\prime}\to F^{\prime}$ are the natural projections, then we have by the above $$[C_{F^{\prime}/\mathfrak{F}^{\prime}}]=[\pi_1^*C_{F^{\prime}/G^{\prime}}]\in A_*(p^*\mathfrak{E}).$$ From the construction of Gysin pull-backs $$[C_{F^{\prime}/G^{\prime}}]=\pi_2^*f^![G^{\prime}]\in A_*(p^*\mathfrak{E}^{\prime})$$ and $$[C_{F^{\prime}/\mathfrak
{F}^{\prime}}]=(\pi_2\circ\pi_1)^*(0\circ f)^![\mathfrak{F}^{\prime}]\in A_*(p^*\mathfrak{E}).$$ Combining the three equalities we get equality (\ref{fult}) above, and therefore the conclusion.
\end{proof}

\begin{theorem}\label{functoriality}
(Functoriality)
 Consider a fibre diagram
\begin{equation*}
 \xymatrix {F^{\prime} \ar[r]^{f^{\prime}}\ar[d]_{p} & G^{\prime}\ar[d]^q\ar[r]^{g^{\prime}}&\mathfrak{M}^{\prime}\ar[d]^r\\
             F\ar[r]^f & G\ar[r]^g&\mathfrak{M}.}
\end{equation*}
Let us assume $f$, $g$ and $g\circ f$ are DM-type morphisms and have perfect relative obstruction theories $E^{\prime\bullet}$, $E^{\prime\prime\bullet}$ and $E^{\bullet}$ respectively and let us denote the associated vector bundle stacks by $\mathfrak{E}^{\prime}$, $\mathfrak{E}^{\prime\prime}$ and $\mathfrak{E}$ respectively. If $F^{\prime},\ F^{\prime}$ admit stratifications by global quotients and $(E^{\prime\bullet}, E^{\prime\prime\bullet},E^{\bullet})$ is a compatible triple, then for any $\alpha\in A_k(\mathfrak{M}^{\prime})$ $$(g\circ f)_{\mathfrak{E}}^!(\alpha)=f^!_{\mathfrak{E}^{\prime}}(g^!_{\mathfrak{E}^{\prime\prime}}(\alpha)).$$
\end{theorem}
\begin{proof}
We argue as in the proof of Theorem 1 in \cite{K:01} (or Theorem 6.5 of \cite{f}).
\\In the same way as in the proof of the previous lemma $\mathfrak{M}^{\prime}$ may be assumed irreducible and reduced and $\alpha=[\mathfrak{M}^{\prime}]$.
\\
\\Consider the vector bundle stacks: $\rho: p^*\mathfrak{E}\to F^{\prime}$, $\pi: q^*\mathfrak{E}^{\prime\prime}\to G^{\prime}$ and $\sigma: \mathfrak{E}^{\prime}\oplus p^*f^*\mathfrak{E}^{\prime\prime}\to F^{\prime}$.
\\By definition 
\begin{align*}
 (g\circ f)^!\mathfrak{M}^{\prime}=&(\rho^*)^{-1}([C_{F^{\prime}/\mathfrak{M}^{\prime}}])\\
g^!\mathfrak{M}^{\prime}=&(\pi^*)^{-1}([C_{G^{\prime}/\mathfrak{M}^{\prime}}]).
\end{align*}
Let us now look at the cartesian diagram
\begin{equation*}
 \xymatrix {F^{\prime} \ar[r]^{f^{\prime}}\ar[d] & G^{\prime}\ar[d]\ar[r]&C_{G^{\prime}/\mathfrak{M}^{\prime}}\ar[d]\\
             F\ar[r]^{f} & G\ar[r]^-0&\mathfrak{E}_{G/\mathfrak{M}}.}
\end{equation*}
From the definition of the pull-back we know that $f^!(g^!\mathfrak{M}^{\prime})$ is equal to $f^!(0^![C_{G^{\prime}/\mathfrak{M}^{\prime}}])$ and by the previous lemma $$f^!(0^![C_{G^{\prime}/\mathfrak{M}^{\prime}}])=(0\circ f)^![C_{G^{\prime}/\mathfrak{M}^{\prime}}].$$ If we denote $C_{G^{\prime}/\mathfrak{M}^{\prime}}$ by $C_0$, then the above shows that $ f^!(g^!\mathfrak{M}^{\prime})$ is represented in $\mathfrak{E}^{\prime}\oplus f^*\mathfrak{E}^{\prime\prime}$ by the cycle $[C_{F^{\prime}/C_{0}}]$. The construction respects equivalence in Chow groups and so we are reduced to showing
\begin{equation}\label{toshow}
 (\sigma^*)^{-1}([C_{F^{\prime}/C_0}])=(\rho^*)^{-1}([C_{F^{\prime}/\mathfrak{M}^{\prime}}])
\end{equation}
in $A_*F^{\prime}$.
\\
\\Introduce the double deformation space $M^{\prime}:=M^{\circ}_{F^{\prime}\times\pp^1/M^{\circ}_{G^{\prime}/\mathfrak{M}^{\prime}}}\to\pp^1\times\pp^1$ with general fiber $M^{\circ}_{G^{\prime}/\mathfrak{M}^{\prime}}$ and special fibre $C_{F^{\prime}\times\pp^1/M^{\circ}_{G^{\prime}/\mathfrak{M}^{\prime}}}$ over $\{0\}\times \pp^1$ (see \cite{K:01}, proof of Theorem 1). Restricting to this special fibre and considering the rational equivalence on the second $\pp^1$ we see that
\begin{equation}\label{equiv}
[C_{F^{\prime}/C_{0}}]\sim[C_{F^{\prime}/\mathfrak{M}^{\prime}}]
\end{equation}
 in $A_*(C_{F^{\prime}\times\pp^1/M^{\circ}_{G^{\prime}/\mathfrak{M}^{\prime}}})$.
\\In a completely analogous fashion there exists a double deformation space $M:=M^{\circ}_{F\times\pp^1/M^{\circ}_{G/\mathfrak{M}}}$. If we consider the map $w:F^{\prime}\times\pp^1\stackrel{p\times 1_{\pp^1}}\rightarrow F\times\pp^1$, then the general fibers of $M$ and $M^{\prime}$ are related by the cartesian diagram
\begin{equation*}
 \xymatrix{F^{\prime}\times\pp^1 \ar[r]\ar[d]_{w}&M^{\circ}_{G^{\prime}/\mathfrak{M}^{\prime}}\ar[d]\\
F\times\pp^1 \ar[r]&M^{\circ}_{G/\mathfrak{M}}.}
\end{equation*}
This implies  $C_{F^{\prime}\times\pp^1/M^{\circ}_{G^{\prime}/\mathfrak{M}^{\prime}}}\stackrel{i}\rightarrow (p\times 1_{\pp^1})^*C_{F\times\pp^1/M^{\circ}_{G/\mathfrak{M}}}$ is a closed immersion and consequently we can push forward relation (\ref{equiv}) in $A_*(w^*C_{F\times\pp^1/M^{\circ}_{G/\mathfrak{M}}})$.
Now, by Proposition 1, in \cite{K:01}, we have a morphism $$A_*(C_{F\times\pp^1/M^{\circ}_{G/\mathfrak{M}}})\stackrel{i_*}{\hookrightarrow} A_*(h^1/h^0(c(u)^{\vee}))$$ where $u:=(T\cdot id, U\cdot can)$ is the map
\begin{equation*}
 f^*L_{G/\mathfrak{M}}\otimes\mathcal{O}_{\pp^1}(-1)\stackrel{u}{\rightarrow}f^*L_{G/\mathfrak{M}}\oplus L_{F/\mathfrak{M}}
\end{equation*}
in $\mathcal{D}(F\times\pp^1)$ and $c(u)$ its mapping cone. Here we denoted by $T$ and $U$ the homogeneous coordinates on $\pp^1$. Let us consider the closed immersion $w^*i:A_*(w^*C_{F\times\pp^1/M^{\circ}_{G/\mathfrak{M}}})\stackrel{}{\hookrightarrow} A_*(w^*h^1/h^0(c(u)^{\vee}))$. Then pushing forward via $w^*i$ the equivalence relation we have in $A_*(w^*C_{F\times\pp^1/M^{\circ}_{G/\mathfrak{M}}})$, we obtain the equivalence relation (\ref{equiv}) in $A_*(w^*h^1/h^0(c(u)^{\vee}))$.
\\Let us now use the notation of Construction \ref{constr}. Consider the morphism $v:=(T\cdot id, U\cdot\varphi):f^*E_{G/\mathfrak{M}}\otimes\mathcal{O}_{\pp^1}(-1)\to f^*E^{\prime\prime}\oplus E^{\prime}$ in $\mathcal{D}(F\times\pp^1)$. The morphism of distinguished triangles in Definition \ref{compatib} gives a morphism of distinguished triangles 
 \begin{equation*}\small
 \xymatrix {(fp)^*E^{\prime\prime\bullet}(-1) \ar[r]^-{w^*v}\ar[d] & (fp)^*E^{\prime\prime\bullet}\oplus p^*E^{\bullet}\ar[d]\ar[r]&w^*c(v)\ar[d]\ar[r]&(fq)^*E^{\prime\prime\bullet}(-1)[1]\ar[d]\\
             (fp)^*L_{G/\mathfrak{M}}(-1)\ar[r]^-{w^*u} &(fp)^*L_{G/\mathfrak{M}}\oplus p^*L_{F/\mathfrak{M}}\ar[r]&w^*c(u)\ar[r]&(fq)^*L_{G/\mathfrak{M}}(-1)[1]}
\end{equation*}
over $F^{\prime}\times\pp^1$. Dualizing and taking $h^1/h^0$ of the map $w^*c(v)\to w^*c(u)$, we obtain a morphism of Picard stacks $w^*h^1/h^0(c(u)^{\vee})\to w^*h^1/h^0(c(v)^{\vee})$ that is a closed immersion. Therefore, we can push forward the rational equivalence (\ref{equiv}) on $w^*h^1/h^0(c(v)^{\vee})$ that is a vector bundle stack on $F^{\prime}\times\pp^1$. The fact that the above map between cone stacks is a closed immersion follows from Prop 2.6 in \cite{bf} and the fact that the maps in cohomology induced by the vertical maps in the above diagram are isomorphisms in degree $0$ and surjective in degree $-1$.
\\
\\Let us now conclude the proof. We have obtained $[C_{F^{\prime}/C_{0}}]\sim[C_{F^{\prime}/\mathfrak{M}^{\prime}}]$ in $A_*(w^*h^1/h^0(c(v)^{\vee}))$. Looking at $w^*h^1/h^0(c(v)^{\vee})\to\pp^1$, we see that $w^*h^1/h^0(c(v)^{\vee})$ restricts to $F_0:=p^*\mathfrak{E}^{\prime}\oplus p^*f^*\mathfrak{E}^{\prime\prime}$ and $F_1:=p^*\mathfrak{E}$ in $F^{\prime}\times\{0\}$ respectively $F^{\prime}\times\{1\}$. Consider the map
\begin{equation*}
A_*(C_{F^{\prime}\times\pp^1/M^{\circ}_{G^{\prime}/\mathfrak{M}^{\prime}}})\to A_*(w^*h^1/h^0(c(v)^{\vee}))\to A_*(F_i)\to F^{\prime}.
\end{equation*}
We have that the image of $[C_{F^{\prime}\times\pp^1/M^{\circ}_{G^{\prime}/\mathfrak{M}^{\prime}}}]$ in $A_*(F_0)$ is $[C_{F^{\prime}/C_0}]$ and in $A_*(F_1)$ is $[C_{F^{\prime}/\mathfrak{M}^{\prime}}]$. As the composition does not depend on $i$ we deduce equality (\ref{toshow}).
\end{proof}

\begin{corollary}\label{virtual}
Let us assume we have a commutative diagram
\begin{equation*}
 \xymatrix{F\ar[rr]^f\ar[rd]_{\epsilon_F}&&G\ar[ld]^{\epsilon_{G}}\\
&\mathfrak{M}}
\end{equation*}
with $\mathfrak{M}$ of pure dimension. If $F$ and $G$ admit stratifications by global quotients and  $(E^{\bullet}_{F/G},E^{\bullet}_{G/\mathfrak{M}},E^{\bullet}_{F/\mathfrak{M}})$ a compatible triple, then $$f_{\mathfrak{E}_{F/G}}^![G]^{\vv}=[F]^{\vv}.$$
\end{corollary}
\begin{proof}
By the definition of virtual classes we have
\begin{align*}
 [F]^{\vv}&=(\epsilon_F)_{\mathfrak{E}_{F/\mathfrak{M}}}^![\mathfrak{M}]\\
[G]^{\vv}&=(\epsilon_G)_{\mathfrak{E}_{G/\mathfrak{M}}}^![\mathfrak{M}].
\end{align*}
Moreover, by the construction of $E_{F/G}$ we are in the hypotheses of Theorem \ref{functoriality} and therefore
\begin{equation*}
 (\epsilon_{G}\circ f)_{\mathfrak{E}_{F/\mathfrak{M}}}^![\mathfrak{M}]=f_{\mathfrak{E}_{F/G}}^!(\epsilon_{G})_{\mathfrak{E}_{G/\mathfrak{M}}}^![\mathfrak{M}].
\end{equation*}
The two equations above show that $f_{\mathfrak{E}_{F/G}}^![G]^{\vv}=[F]^{\vv}.$

\end{proof}
\begin{remark}
Let us consider a cartesian diagram of DM stacks
\begin{equation*}
 \xymatrix {F^{\prime} \ar[r]\ar[d]_{g} & G^{\prime}\ar[d]^f\\
             F\ar[r]^i & G}
\end{equation*}
with obstruction theories $E_{F}^{\bullet}$, $E_{G}^{\bullet}$, $E_{F^{\prime}}^{\bullet}$, $E_{G^{\prime}}^{\bullet}$. Let us assume we have perfect obstruction theories $E_{F/G}^{\bullet}$ and $E_{F^{\prime}/G^{\prime}}^{\bullet}$ compatible with  $E_{F}^{\bullet}$, $E_{G}^{\bullet}$ and $E_{F^{\prime}}^{\bullet}$, $E_{G^{\prime}}^{\bullet}$. If $g^*\mathfrak{E}_{F/G}=\mathfrak{E}_{F^{\prime}/G^{\prime}}$, then $i^![G^{\prime}]^{\vv}=[F^{\prime}]^{\vv}$.
\\This generalizes Proposition 5.10 in \cite{bf} and Theorem 1 in \cite{K:01}.
\end{remark}

\section{Applications}
In this section we collect some applications of the virtual pull-back we defined. We take the ground field to be $\C$. By a homology class of a curve we will mean an element of $A_1^{alg}(X)$-- the group of 1-cycles modulo algebraic equivalence (see \cite{f}, Chapter 10). We will shortly denote it by $A_1(X)$.
\subsection{Preliminaries}
Let us fix notations. Let $X$ be a smooth projective variety and $\beta\in A_1(X)$ a homology class of a curve in $X$. We denote by $\overline{M}_{g,n}(X,\beta)$ the moduli space of stable genus-$g$, $n$-pointed maps to $X$ of homology class $\beta$ (see \cite{FP:97}). Let $\epsilon_X: \overline{M}_{g,n}(X,\beta)\to \mathfrak{M}_{g,n}$ be the morphism that forgets the map (and does not stabilize the pointed curve) and $\pi_X: \overline{M}_{g,n+1}(X,\beta)\to \overline{M}_{g,n}(X,\beta)$ the morphism that forgets the last marked point and stabilizes the result. Then it is a well-known fact that  $$E_{\overline{M}_{g,n}(X,\beta)/\mathfrak{M}_{g,n}}^{\bullet}:=(\mathcal{R}^{\bullet}{\pi_{X}}_*ev_{X}^*T_{X})^{\vee}$$ defines an obstruction theory for the morphism $p$, where $ev_X$ indicates the evaluation map $ev_X:\overline{M}_{g,n+1}(X,\beta)\to X$ (see \cite{b}). Unless otherwise stated the map $\pi_X$ will be endowed with this obstruction theory (and not some other). We call $$[\overline{M}_{g,n}(X,\beta)]^{\vv}:=(\epsilon_X)_{\mathfrak{E}_{\overline{M}_{g,n}(X,\beta)/\mathfrak{M}}}^![\mathfrak{M}_{g,n}]$$ the virtual class of $\overline{M}_{g,n}(X,\beta)$. The dimension of $[\overline{M}_{g,n}(X,\beta)]^{\vv}$ is called the virtual dimension of $\overline{M}_{g,n}(X,\beta)$ and we denote it by $\vdim\overline{M}_{g,n}(X,\beta)$. To a collection of Chow (or cohomology) classes $\gamma_i\in A^{k_i}(X)$ such that $\sum_{i=1}^n k_i=\vdim \overline{M}_{g,n+1}(X,\beta)$, one can associate a Gromov-Witten (shortly GW) invariant defined to be $$I_{g,n,\beta}^X:=\prod_{i=1}^nev_i^*\gamma_i\cdot[\overline{M}_{g,n}(X,\beta)]^{\vv}.$$
\begin{remark}
Let $f: X\to Y$ be a morphism of smooth algebraic varieties. Let $\beta\in A_1(X)$ and $g,\ n$ be any natural numbers. Then $f$ induces a morphism of stacks $\bar{f}: \overline{M}_{g,n}(X,\beta)\to \overline{M}_{g,n}(Y,f_*\beta)$.
\\Convention: Given a morphism of smooth algebraic varieties $f:X\to Y$, we will indicate the induced morphism between moduli spaces of stable maps by the same letter with a bar.
\end{remark}
Let us now state the version of Cohomology and Base Change we will use in our applications. We refer to \cite{gm}, III.8.
\begin{theorem} (Cohomology and Base Change)\label{grothmaps} Let $G^{\prime}\stackrel{p}{\rightarrow}G$ be a flat morphism of separated DM stacks and $E^{\bullet}$ a finite bounded of locally free sheaf on $G^{\prime}$. Then for any base change
\begin{equation*}
 \xymatrix{F^{\prime}\ar[r]^{g}\ar[d]_q&G^{\prime}\ar[d]^p\\
F\ar[r]^f&G}
\end{equation*}
we have a canonical isomorphism in $\mathcal{D}_F$ $$Lf^*Rp_*E^{\bullet}\simeq Rq_*Lg^*E^{\bullet}.$$
\end{theorem}

%\begin{corollary}\label{grothmaps} Let $f:X\to\pp$ be a morphism of smooth projective schemes and let $T^{\bullet}:=[T^0\to T^1]$ be a two term complex of bundles in $\mathcal{D}_{\pp}$ concentrated in $[0,1]$. Then we have the following canonical isomorphism in $\mathcal{D}_{\overline{M}_{0,n}(X,\beta)}$ $$\bar{f}^*\mathcal{R}^{\bullet}{\pi_{\pp}}_*ev_{\pp}^*T^{\bullet}\simeq\mathcal{R}^{\bullet}{\pi_X}_*ev_{X}^*f^*T^{\bullet}.$$ \end{corollary}
%\begin{proof} By Proposition 3.7.2 in \cite{lip}, there is a canonical homomorphism $$\psi:\bar{f}^*\mathcal{R}^{\bullet}{\pi_{\pp}}_*ev_{\pp}^*T^{\bullet}\simeq\mathcal{R}^{\bullet}{\pi_X}_*ev_{X}^*f^*T^{\bullet}.$$ Using Cech complex resolutions one can show that $\psi$ is an isomorphism. \end{proof}
\begin{proposition}\label{cbg}
 Let $f:X\rightarrow \pp$ be a morphism of smooth projective varieties and let $T_{X/\pp}$ be the dual of the cotangent complex of $X$ to $\pp$. Then, in notations as above
\\(i) The map $\bar{f}:\overline{M}_{g,n}(X,\beta)\to \overline{M}_{g,n}(\pp,f_*\beta)$ has a dual obstruction theory $E_{\overline{M}_{g,n}(X,\beta)/\overline{M}_{g,n}(\pp,f_*\beta)}^{\bullet}$ isomorphic to $\mathcal{R}^{\bullet}{\pi_X}_*ev_X^*T_{X/\pp}$ in $\mathcal{D}_{\overline{M}_{g,n}(X,\beta)}$.
\\(ii) If $f=i$ is an embedding and $N_{X/\pp}$ denotes the normal bundle of $X$ in $\pp$, then $\bar{i}:\overline{M}_{g,n}(X,\beta)\to \overline{M}_{g,n}(\pp,f_*\beta)$ has a dual obstruction theory $(E_{\overline{M}_{g,n}(X,\beta)/\overline{M}_{g,n}(\pp,i*\beta)}^{\bullet})^{\vee}$ $$[0\to\mathcal{R}^0{\pi_X}_*ev_X^* N_{X/\pp}\to \mathcal{R}^1{\pi_X}_*ev_X^* N_{X/\pp}]$$ in $[0,2]$.
\\(iii) If $g=0$ and $\pp$ is convex, then $E_{\overline{M}_{g,n}(X,\beta)/\overline{M}_{g,n}(\pp,f_*\beta)}^{\bullet}$ is perfect.
\end{proposition}
\begin{proof}
(i) In notations as in the beginning of the section, the relative obstruction theories are $E_{\overline{M}_{0,n}(\pp, f_*\beta)/\mathfrak{M}}^{\bullet}:=(\mathcal{R}^i{\pi_{\pp}}_*ev_{\pp}^*T_{\pp})^{\vee}$ and $E_{\overline{M}_{0,n}(X, \beta)/\mathfrak{M}}^{\bullet}:=(\mathcal{R}^i{\pi_X}_*ev_X^*T_{X})^{\vee}$. The distinguished triangle
\begin{equation}\label{nor}
T_{X}\to f^*T_{\pp}\to T_{X/\pp}\to T_X[1]                                         
\end{equation}
 induces a distinguished triangle $$\mathcal{R}^{\bullet}{\pi_X}_*ev_X^*T _X\stackrel{\varphi}{\rightarrow} \mathcal{R}^{\bullet}{\pi_X}_*ev_X^*f^*T_{\pp}\to \mathcal{R}^{\bullet}{\pi_X}_*ev_X^*T_{X/\pp}\to \mathcal{R}^{\bullet}{\pi_X}_*ev_X^*f^*T_X[1].$$ We now need to show that
\begin{equation*}
\mathcal{R}^{\bullet}{\pi_X}_*ev_X^*f^*T_{\pp}=\bar{f}^*\mathcal{R}^{\bullet}{\pi_{\pp}}_*ev_{\pp}^*T_{\pp}
\end{equation*}
 in the derived category of $\overline{M}_{0,n}(X,\beta)$ and this follows by Theorem \ref{grothmaps}.
\\(ii) If $f=i$ is an embedding, then $T_{X/\pp}$ is quasi-isomorphic to $[0\to N_{X/Y}]$ in degrees $[0,1]$ and the claim follows by (i).
\\(iii)  If $\pp$ is convex, then $\overline{M}_{0,n}(\pp, i_*\beta)$ is unobstructed and the claim follows from Remark \ref{smooth}.
\end{proof}
\begin{proposition}\label{cartesian}
Let
\begin{equation}\label{hopefullycart}
 \xymatrix{X^{\prime}\ar[r]^j\ar[d]^p&\pp^{\prime}\ar[d]^q
\\X\ar[r]^i&\pp}
\end{equation}
be a cartesian diagram of smooth projective varieties and let $\beta\in A_1(X^{\prime})$ be any homology class of a curve.
\\(i) Then the induced diagram of moduli spaces of stable maps
\begin{equation*}
 \xymatrix{\overline{M}_{g,n}(X^{\prime},\beta)\ar[r]^{\bar{j}}\ar[d]^{\bar{p}}&\overline{M}_{g,n}(\pp^{\prime},j_*\beta)\ar[d]^{\bar{q}}
\\\overline{M}_{g,n}(X,p_*\beta)\ar[r]^{\bar{i}}&\overline{M}_{g,n}(\pp, (i\circ p)_*\beta)}
\end{equation*}
is commutative. If $i$ is a closed embedding, then it induces an open and closed embedding of $\overline{M}_{g,n}(X^{\prime},\beta)$ in the fiber product $\overline{M}_{g,n}(X,p_*\beta)\times_{\overline{M}_{g,n}(\pp, (i\circ p)_*\beta)}\overline{M}_{g,n}(\pp^{\prime},j_*\beta)$.
\\(ii)If the natural map $p^*L_{X/\pp}\to L_{X^{\prime}/\pp^{\prime}}$ is an isomorphism, then it induces an isomorphism $$E_{\overline{M}_{g,n}(X^{\prime},\beta)/\overline{M}_{g,n}(\pp^{\prime},j_*\beta)}^{\bullet}\simeq\bar{p}^*E_{\overline{M}_{g,n}(X,p_*\beta)/\overline{M}_{g,n}(\pp,(i\circ p)_*\beta)}^{\bullet}.$$ If moreover, $g=0$ and $\pp$ is convex then, $E_{\overline{M}_{g,n}(X^{\prime},\beta)/\overline{M}_{g,n}(\pp^{\prime},j_*\beta)}^{\bullet}$ is perfect.
\end{proposition}
\begin{proof}
Let us prove that the cartesian product is isomorphic to a disjoint union of components corresponding to all homology classes $\eta\in A_1(X^{\prime})$ such that $j_*\eta=j_*\beta$ and $p_*\eta=p_*\beta$. Let $P=\overline{M}_{g,n}(X,p_*\beta)\times_{\overline{M}_{g,n}(\pp, (i\circ p)_*\beta)}\overline{M}_{g,n}(\pp^{\prime},j_*\beta)$. From the universal property of cartesian products we obtain a map $$\psi:\overline{M}_{g,n}(X^{\prime},\beta)\to P.$$ Conversely, let $(f_1:C_X\to X, f_2:C_{\pp^{\prime}}\to\pp^{\prime})\in P$. Let $C_{\pp^{\prime}}^{stab}$ be the source curve of the stabilization of the composite map $C_{\pp^{\prime}}\to\pp^{\prime}\to\pp$. As $i$ is an embedding we have that $C_X$ is isomorphic to  $C_{\pp^{\prime}}^{stab}$. We have thus obtained a natural map $C_{\pp^{\prime}}\to C_X\to X$. From the universal property of cartesian products we obtain a unique map $f:C_{\pp^{\prime}}\to X^{\prime}$ commuting with $f_1$ and $f_2$. Let $\eta=f_*[C_{\pp^{\prime}}]$. Then $\eta$ satisfies  $j_*\eta=j_*\beta_1$ and $p_*\eta=p_*\beta_2$. Let $P_{\beta}$ define the locus in $P$ which by the above procedure induces a map $f:C_{\pp^{\prime}}\to X^{\prime}$ such that $f_*[C]=\beta$. The above translates into the existence of a map $\phi:P_{\beta}\to \overline{M}_{g,n}(X^{\prime},\beta)$. As $\psi$ and $\phi$ are inverse to each other, the statement follows. 
\\The proof of part (ii) follows by Proposition \ref{cbg} (i) applied to the morphism $i$, followed by Theorem \ref{grothmaps} with $f=p$ and $T:=L^{\vee}_{X/\pp}$.
\\If $g=0$ and $\pp$ is convex then $\overline{M}_{g,n}(\pp, (i\circ p)_*\beta)$ is unobstructed and therefore $E_{\overline{M}_{g,n}(X,p_*\beta)/\overline{M}_{g,n}(\pp,(i\circ p)_*\beta)}^{\bullet}$ is perfect. The claim now follows from the first part of the proof.
\end{proof}

\begin{example}
Let us consider a cartesian diagram of smooth projective varieties as above. The induced commutative diagram of moduli spaces of stable maps will not be cartesian in general. One counterexample is the case of $\pp$ is a point $P:=Spec\ k$, $\pp^{\prime}:=Y$ and $X^{\prime}:=X\times Y$. The diagram 
\begin{equation*}
 \xymatrix{\overline{M}_{g,n}(X\times Y,(\beta_1, \beta_2))\ar[r]^-{\bar{p_2}}\ar[d]^{\bar{p_1}}&\overline{M}_{g,n}(Y,\beta_2)\ar[d]
\\\overline{M}_{g,n}(X,\beta_1)\ar[r]&\overline{M}_{g,n}(P, 0)}
\end{equation*}
is \emph{not} cartesian. Let $M$ be the cartesian product $\overline{M}_{g,n}(Y,\beta_2)\times_{\overline{M}_{g,n}(P, 0)}\overline{M}_{g,n}(X,\beta_1)$. Then, we still have a nice relation between the virtual class of $M$ and the virtual class of $\overline{M}_{g,n}(X\times Y,(\beta_1, \beta_2))$. This was studied by Behrend (see \cite{b2}, Theorem 1).
\end{example}
 
\begin{example}
Let
\begin{equation*}
 \xymatrix{X^{\prime}\ar[r]^j\ar[d]^p&\pp^{\prime}\ar[d]^q
\\X\ar[r]^i&\pp}
\end{equation*}
be a cartesian diagram of smooth projective varieties as in the above proposition. Let us moreover suppose that $i$ and $j$ induce injective morphisms of groups $i_*:A_1(X)\to A_1(\pp)$, respectively $j_*:A_1(X^{\prime})\to A_1(\pp^{\prime})$ (but $i$ is not supposed to be injective!). Then, the corresponding commutative diagram between moduli spaces of stable maps is cartesian.
\begin{proof}
 Let us fix a scheme $S$ and let us consider $f:C_{\pp}\to \pp$ an element in $\overline{M}_{g,n}(\pp,(i\circ p)_*\beta)(S)$. As before, let us consider $f_1:C_{\pp^{\prime}}\to \pp^{\prime}$ an object in $\overline{M}_{g,n}(\pp^{\prime},j_*\beta)(S)$ and $f_2:C_{X}\to X$ an object in $\overline{M}_{g,n}(X,p_*\beta)(S)$ which both map to $f:C_{\pp}\to \pp$. We have that $C_{\pp}$ is canonically isomorphic to $C_{\pp^{\prime}}^{stab}$ the source curve of the stabilization of the composite map $C_{\pp^{\prime}}\to\pp^{\prime}\to\pp$. Then, by our hypothesis on $i$ we have that the curve $C_X$ is also canonically isomorphic to $C_{\pp^{\prime}}^{stab}$. We have thus obtained a commutative diagram 
\begin{equation*}
 \xymatrix{C_{\pp^{\prime}}\ar[r]\ar[dd]&\pp^{\prime}\ar[dr]
\\&&\pp
\\C_{\pp^{\prime}}^{stab}\ar[r]&X\ar[ur]}
\end{equation*}
and therefore by the universal property of cartesian products, we have obtained a canonical map $C_{\pp^{\prime}}\to X^{\prime}$. As in the proof of the proposition our hypothesis on $j$ implies that the map $\overline{M}_{g,n}(X^{\prime},\beta)\to P$ is surjective.
\end{proof}

\end{example}

\begin{remark}\label{condand}
 If $i$ is a closed embedding and the natural map $A_1(X^{\prime})\to A_1(\pp^{\prime})\oplus A_1(X)$ is injective then diagram (\ref{hopefullycart}) is cartesian. This follows from the fact that the above condition implies that the natural map  $$\overline{M}_{g,n}(X^{\prime},\beta)\to\overline{M}_{g,n}(X,p_*\beta)\times_{\overline{M}_{g,n}(\pp, (i\circ p)_*\beta)}\overline{M}_{g,n}(\pp^{\prime},j_*\beta)$$ is surjective.
\end{remark}
\begin{example}
Let us now look at an example where our construction of virtual pull-backs does not apply. If we consider $\overline{M}_{g,n}(X,\beta)\stackrel{\bar{i}}{\rightarrow}\overline{M}_{g,n}(\pp, i_*\beta)$ with $\pp$ a convex space, then the construction applies without further conditions only in genus zero. In general (higher genus or non-convex $\pp$) $h^{-2}(E_{\overline{M}_{g,n}(X,\beta)/\overline{M}_{g,n}(\pp, i_*\beta)})$ might not vanish.
\\To see an example, let us consider $i:\pp^r\hookrightarrow\pp^r\times\pp^s$, the inclusion into the first factor. Then we have an induced map $\overline{M}_{g,n}(\pp^r,d_1)\hookrightarrow\overline{M}_{g,n}(\pp^r\times\pp^s,(d_1,0))$. From Corollary \ref{blup} we obtain that the dual relative obstruction theory of $\\bar{i}$ is $\mathcal{R}^{\bullet}\pi_*ev^*N_{\pp^r/\pp^r\times\pp^s}$. We have that the normal bundle $N_{\pp^r/\pp^r\times\pp^s}$ is isomorphic to $\mathcal{O}_{\pp^r}^{\oplus s}$. Since $\mathcal{R}^{1}\pi_*f^*(\mathcal{O}_{\pp^r}^{\oplus s})$ is non-zero for $g\geq1$, the (dual) relative obstruction theory will never be perfect.
\end{example}

\subsection{Pulling back divisors}
Let $\pp$ be a convex variety and $d\in A_1(\pp)$ be the class of a curve. If $X\stackrel{i}\hookrightarrow\pp$ is an embedding of smooth projective varieties, then $i$ induces a morphism $\overline{M}_{0,n}(X,d)\stackrel{\bar{i}}\hookrightarrow\overline{M}_{0,n}(\pp,d)$ where we made the convention that $\overline{M}_{0,n}(X,d)$ is the union of all $\overline{M}_{0,n}(X,\beta)$ such that $i_*\beta=d$.  Let $D_{\pp}:=D_{\pp}(0,n_1,d_1\mid 0,n_2,d_2)$ be a boundary divisor in $\overline{M}_{0,n}(\pp,d)$ that comes with a virtual class obtained by pull-back along the obvious forgetful morphism
\begin{equation*}
 D_{\pp}\to\mathfrak{M}_{0,n_1+1}\times\mathfrak{M}_{0,n_2+1}
\end{equation*}
and analogously we have a boundary divisor $D_{X}:=D_{X}(0,n_1,d_1\mid 0,n_2,d_2)$ in $\overline{M}_{0,n}(X,d)$ equipped with a virtual fundamental class. Constructing the following cartesian diagram
\begin{equation*}
\xymatrix{D_{X}\ar[r]\ar[d]&D_{\pp}\ar[r]\ar[d]&\mathfrak{M}_{0,n_1+1}\times\mathfrak{M}_{0,n_2+1}\ar[d]\\
\overline{M}_{0,n}(X,d)\ar[r]^{\bar{i}}&\overline{M}_{0,n}(\pp,d)\ar[r]&\mathfrak{M}}
\end{equation*}
we get
 \begin{align*}
\bar{i}^![D_{\pp}(0,n_1,d_1\mid 0,n_2,d_2)]^{\vv}=[D_{X}(0,n_1,d_1\mid 0,n_2,d_2)]^{\vv}.
\end{align*}
Indeed, it is easy to check that the obstructions are compatible.
\begin{remark}
One could na\"ively hope to obtain new relations between the rational GW invariants of $X$ by pulling back the WDVV relations in $\pp$. The above shows that for any $X\hookrightarrow\pp$ pulling-back the WDVV equations in $\overline{M}_{0,n}(\pp,d)$ gives the WDVV equations in $\overline{M}_{0,n}(X,d)$.
\end{remark}

\subsection{Blow-ups}\label{blup}

Let $X$ be a smooth $r$-dimensional projective variety, $Y\subseteq X$ a smooth $r^{\prime}$-codimensional subvariety and $p_X:\tilde{X}\to X$ the blow-up of $X$ in $Y$, with exceptional divisor $E$.

\begin{definition} \label{lifting}
For a blow up $p_X:\tilde{X}\to X$ and a class $\beta\in A_1(X)$ we call the class $p_X^!\beta$ the lifting of $\beta$ and we denote it by $\tilde{\beta}$, where $p_X^!$ is the refined intersection product of \cite{f}, Chapter 8.
\end{definition}

\begin{remark} The lifting of $\beta$ satisfies two basic properties that follow trivially from the projection formula, namely $(p_X)_*\tilde{\beta}=\beta$ and $\tilde{\beta}\cdot E=0$.
 
\end{remark}
\begin{lemma}
The moduli space of stable maps to $\tilde{X}$ of class $\tilde{\beta}$ and the moduli space of stable maps to $X$ of class $\beta$ have the same virtual dimension.
\end{lemma}
\begin{proof}
By \cite{f} we know that 
\begin{equation*}
K_{\tilde{X}}=p_X^*K_{X}+(r^{\prime}-1)E
\end{equation*}
and therefore the virtual dimension of $\overline{M}_{g,n}(\tilde{X},\tilde{\beta})$ is
\begin{align*} 
\vdim(\overline{M}_{g,n}(\tilde{X},\tilde{\beta}))&=(1-g)(r-3)-K_{\tilde{X}}\cdot\tilde{\beta}+n
\\&=(1-g)(r-3)-[p_X^*K_{X}+(r^{\prime}-1)E]\cdot\tilde{\beta}+n
\\&=(1-g)(r-3)-p_X^*K_{X}\cdot\tilde{\beta}+n
\\&=(1-g)(r-3)-K_{X}\cdot\beta+n
\\&=\vdim(\overline{M}_{g,n}(X,\beta)).
 \end{align*}

\end{proof}

\begin{lemma}\label{andreas}
 In notations as before, the natural projection $p_{X}:\tilde{X}\to X$ induces a morphism $\bar{p}_{X}:\overline{M}_{0,n}(\tilde{X},\tilde{\beta})\to\overline{M}_{0,n}(X,\beta)$. If $X=\pp$ is a homogeneous space and $\tilde{X}=\tilde{\pp}$ the blow up of $\pp$, then $$({\bar{p}}_{\pp})_*[\overline{M}_{0,n}(\tilde{\pp},\tilde{\beta})]^{\vv}=[\overline{M}_{0,n}(\pp,\beta)]^{\vv}.$$
\end{lemma}
\begin{proof}
 The proof is a straightforward generalization of \cite{G:01}, Proposition 2.2. Let us write the proof for completeness. Since $\pp$ is convex the stack $\overline{M}_{0,n}(\pp,\beta)$ is smooth of expected dimension $d$. As $\overline{M}_{0,n}(\pp,\beta)$ is connected (see \cite{kp}) it follows taht $\overline{M}_{0,n}(\pp,\beta)$ is also irreducible. This shows that
 \begin{equation*}
(\bar{p}_{\pp})_*[\overline{M}_{0,n}(\tilde{\pp},\tilde{\beta})]^{\vv}=\alpha[\overline{M}_{0,n}(\pp,\beta)]
\end{equation*}
for some $\alpha\in\Q $. If we show that $\bar{p}_{\pp}$ is a local isomorphism around a generic point $\mathcal{C}:=(C, x_1,...,x_n,f)\in Z_i$ for any$1\leq i\leq k$ then by \cite{G:01} we have $$(\bar{p}_{\pp})_*[\overline{M}_{0,n}(\tilde{\pp},\tilde{\beta})]^{\vv}=\overline{M}_{0,n}(\pp,\beta).$$
Let us denote the locus $\mathcal{C}\in\overline{M}_{0,n}(\pp,\beta)$ of maps such that $\bar{p}^{-1}(\mathcal{C})$ is not a point by $M$. Then $M$ is embedded in $M^{\prime}$ the locus of maps to $\pp$ such that $f(C)$ intersects the blown up locus $Y$. Let us analyze the dimension of $M^{\prime}$. By Kleiman's transversality theorem we have that the locus of smooth curves in $\pp$ which intersect a given subvariety $Y$ of codimension $c$ has the expected dimension $d+1-c$.  Using fact that $Y$ has codimension at least two, we obtain that $M^{\prime}$ this dimension is at most $d-1$. This shows that $\bar{p}$ is an isomorphism around the generic point $\mathcal{C}\in\overline{M}_{0,n}(\pp,\beta)$.  
\end{proof}
\begin{proposition}\label{main} Let $X$ be a smooth projective subvariety of some homogeneous space $\pp$ and let us assume that there exists $Z$ a smooth subvariety of $\pp$, such that $X$ and $Z$ intersect transversely. Let $Y:=X\cap Z$ and $\tilde{X}$ be the blow-up of $X$ along $Y$. Then for any non-negative integer $n$ and any $\beta\in A_1(X)$ with lifting $\tilde{\beta}\in A_*(\tilde{X})$ $$(\bar{p}_{X})_*[\overline{M}_{0,n}(\tilde{X},\tilde{\beta})]^{\vv}=[\overline{M}_{0,n}(X,\beta)]^{\vv}.$$
\end{proposition}
\begin{proof}
If $Y=X\cap Z$, $\tilde{X}$ is the blow-up of $X$ along $Y$ and $\tilde{\pp}$ is the blow-up of $\pp$ along $Z$
then the diagram 
 \begin{equation*}
 \xymatrix {\tilde{X} \ar[r]^{j}\ar[d]_{p_X} &\tilde{\pp}\ar[d]^{p_{\pp}}\\
             X\ar[r]^i & \pp}
\end{equation*}
is cartesian. By Theorem \ref{relations} (ii) we have that 
\begin{equation}\label{injhomology}
p_{\pp}^!i_*\beta=j_*p_{X}^!\beta.
\end{equation}
Let us take $\delta\in A_1(\tilde{X})$ such that $({p_X}_*\delta,j_*\delta)=(\beta, p_{\pp}^!i_*\beta)$. By ${p_X}_*\delta=\beta$ we get that $\delta=p^!\beta+de$, where $d\in\Z$ and $e$ is the class of a curve contained in some fiber of $p_X$. By equality (\ref{injhomology}) we get $j_*\delta=p_{\pp}^!i_*\beta+de$. Using $j_*\delta=p_{\pp}^!i_*\beta$ in the equality above we obtain $d=0$, which shows that the map $({p_X}_*,j_*):A_1(\tilde{X})\to A_1(X)\oplus A_1(\tilde{\pp})$ is injective. Using equality (\ref{injhomology}) and Remark \ref{condand} we obtain a cartesian diagram
 \begin{equation}\label{blmod}
 \xymatrix {\overline{M}_{0,n}(\tilde{X},\tilde{\beta})\ar[r]^{\bar{j}}\ar[d]_{\bar{p}_X} &\overline{M}_{0,n}(\tilde{\pp}\ar[d]^{\bar{p}_{\pp}},\widetilde{i_*\beta})\\
             \overline{M}_{0,n}(X,\beta)\ar[r]^{\bar{i}} & \overline{M}_{0,n}(\pp,i_*\beta).}
\end{equation}
 In order to apply the virtual push-forward machinery to this diagram, we first need to analyze the obstruction theories involved. By Construction \ref{constr} and Corollary \ref{virtual} applied to $\bar{i}$, we have 
\begin{equation}\label{down}
\bar{i}^![\overline{M}_{0,n}(\pp, i_*\beta)]^{\vv}=[\overline{M}_{0,n}(X,\beta)]^{\vv}.
\end{equation}
We know that $p_X^*N_{X/\pp}= N_{\tilde{X}/\tilde{\pp}}$. By Proposition \ref{cartesian} (ii) we obtain that $$E_{\overline{M}_{0,n}(\tilde{X},\tilde{\beta})/\overline{M}_{0,n}(\tilde{\pp}, \widetilde{i_*\beta})}^{\bullet}=\bar{p}_X^*E_{\overline{M}_{0,n}(X,\beta)/\overline{M}_{0,n}(\pp, i_*\beta)}^{\bullet}.$$ This shows in particular that $E_{\overline{M}_{0,n}(\tilde{X},\tilde{\beta})/\overline{M}_{0,n}(\tilde{\pp}, \widetilde{i_*\beta})}^{\bullet}$ is perfect. Applying Corollary \ref{virtual} to $\bar{j}$ we get
\begin{equation*}
\bar{j}^![\overline{M}_{0,n}(\tilde{\pp
},\widetilde{ i_*\beta})]^{\vv}=[\overline{M}_{0,n}(\tilde{X},\tilde{\beta})]^{\vv}
\end{equation*}
and by Proposition \ref{relations} (iii) we obtain that 
\begin{equation}\label{up}\bar{i}^![\overline{M}_{0,n}(\tilde{\pp
},\widetilde{ i_*\beta})]^{\vv}=[\overline{M}_{0,n}(\tilde{X},\tilde{\beta})]^{\vv}.
\end{equation}
Theorem \ref{relations} (i) gives 
\begin{equation}\label{proj}
\bar{i}^!(\bar{p}_{\pp})_*[\overline{M}_{0,n}(\tilde{\pp}, \widetilde{i_*\beta})]^{\vv}=(\bar{p}_X)_*\bar{i}^![\overline{M}_{0,n}(\tilde{\pp}, \widetilde{i_*\beta})]^{\vv}.
 \end{equation}
By Proposition \ref{andreas} we have
\begin{equation}\label{convexcase}
\bar{i}^!(\bar{p}_{\pp})_*[\overline{M}_{0,n}(\tilde{\pp}, \widetilde{i_*\beta})]^{\vv}=\bar{i}^![\overline{M}_{0,n}(\pp, i_*\beta)]^{\vv}.
\end{equation}
Gathering all together, equations (\ref{down}), (\ref{up}), (\ref{proj}), (\ref{convexcase}) translate into
\begin{equation*}
(\bar{p}_X)_*[\overline{M}_{0,n}(\tilde{X}, \tilde{\beta})]^{\vv}=[\overline{M}_{0,n}(X,\beta)]^{\vv}.
\end{equation*}
\end{proof}
The projection formula gives the following Corollary.
\begin{corollary}
 Let $X$ and $Y$ as above, and let $\gamma\in  A^*(X)^{\otimes n}$ be any $n$-tuple of classes such that $\sum codim(\gamma_i)=vdim\overline{M}_{0,n}(X,\beta)$. Then, $I_{0,n,\tilde{\beta}}^{\tilde{X}}(p_X^*\gamma)={I}_{0,n,\beta}^{X}(\gamma)$.
\end{corollary}
\begin{remark}
 The equality in the statement of proposition \ref{main} was obtained in \cite{h} in a more general context, namely under the assumption $N_{Y/X}$ is convex with an extra minor assumption. Lai analyzes the map  $\bar{p}_X:\overline{M}_{0,n}(\tilde{X},\tilde{\beta})\rightarrow\overline{M}_{0,n}(X,\beta)$ and he uses \textit{absolute} obstruction theories (see [ibid.], Section 2). These induce a perfect relative obstruction to $\bar{p}$. Lai analyzes the normal cones of $\overline{M}_{0,n}(\tilde{X},\tilde{\beta})$ and $\overline{M}_{0,n}(X,\beta)$ and he uses the relation between them in order to obtain that $(\bar{p}_X)_*[\overline{M}_{0,n}(\tilde{X}, \tilde{\beta})]^{\vv}=[\overline{M}_{0,n}(X,\beta)]^{\vv}$ (see \cite{h} Theorem 4.11). In our language Lai's assumptions imply that $\bar{p}_X$ admits a virtual pull-back. We should stress however, that we cannot use the usual relative obstruction theories to $\mathfrak{M}_{0,n}$ in order to deduce $\bar{p}^![\overline{M}_{0,n}(X,\beta)]^{\vv}=[\overline{M}_{0,n}(\tilde{X},\tilde{\beta})]^{\vv}$. More precisely, the following diagram 
\begin{equation*}  \xymatrix{\overline{M}_{0,n}(\tilde{X},\tilde{\beta})\ar[rr]^{\bar{p}}\ar[rd]_{\epsilon_{\tilde{X}}}&&\overline{M}_{0,n}(X,\beta)\ar[ld]^{\epsilon_{X}}\\
&\mathfrak{M}_{0,n}}
\end{equation*}
is not commutative. To see this, one can take a map $(C,x_1,...,x_n, f)\in \overline{M}_{0,n}(\tilde{X},\tilde{\beta})$, with $C$ a reducible curve with two components $C_1$ and $C_2$ intersecting in one point and such that $C_1$ has no marked points and it is contracted by $p\circ f$. We have that $$\epsilon_{\tilde{X}}(C,x_1,...,x_n, f)=C$$ while $$\epsilon_X\circ\bar{p}(C,x_1,...,x_n, f)=C_2.$$ This shows that Corollary \ref{virtual} does not apply to the above diagram.
\end{remark}
\begin{remark}
 If $X$ is the zero-locus of a section $s\in H^0(\pp, V)$, for some convex vector bundle $V$ on $\pp$ and $Y$ respects the hypothesis of Proposition \ref{virtual}, then the equality $(\bar{p}_X)_*[\overline{M}_{0,n}(\tilde{X}, \tilde{\beta})]^{\vv}=[\overline{M}_{0,n}(X,\beta)]^{\vv}$ follows from the ``Conjecture'' proved in \cite{K:01} as described below. In notations of [ibid.] we have
\begin{equation*}
 \bar{i}_*[\overline{M}_{0,n}(X,\beta)]^{\vv}=c_{top}(\mathcal{R}^0(\pi_{\pp})_*ev_{\pp}^*V )\cdot[\overline{M}_{0,n}(\pp,i_*\beta)]^{\vv}.
\end{equation*}
Again, using the isomorphism in Proposition \ref{cartesian} (ii) we get the same relation with blow-ups, namely,
\begin{equation*}
 \bar{j}_*[\overline{M}_{0,n}(\tilde{X},\tilde{\beta})]^{\vv}=c_{top}(\bar{p}^*\mathcal{R}^0(\pi_{\tilde{\pp}})_*ev_{\tilde{\pp}}^*V)\cdot [\overline{M}_{0,n}(\tilde{\pp},\widetilde{i_*\beta})]^{\vv}.
\end{equation*}
Now, the equality follows from the projection formula.
\end{remark}

\subsection{Projective bundles}
Let $X$ be a a smooth projective variety of dimension $n$. Let $\C_X$ be the trivial line bundle on $X$ and $V:=W\oplus\C_X$ be a rank $r$ vector bundle on $X$ with non-zero Chern roots $\{c_1,...,c_{r-1}\}$. We denote by $p_X:\pp_X(V)\to X$ the associated projective bundle. It is well known that there exists an isomorphism $\varphi: A_1(\pp_X(V))\to A_1(X)\oplus f\cdot\Z$, where $f_X$ denotes the class of a curve in a fibre of $p_X$. Let us fix such an isomorphism. For this, let us consider the following exact sequence
\begin{equation*}
 0\stackrel{}{\rightarrow} A_1(\pp^{r-1})\stackrel{i}{\rightarrow} A_1(\pp_X(V))\stackrel{p}{\rightarrow} A_1(X)\stackrel{}{\rightarrow} 0
\end{equation*}
where $i$ is the map induced by the inclusion of a fiber of $p_X$ and $p$ denotes the push-forward by $p_X$. Then, taking $s_X:X\to\pp_X(V)$ the zero section of the projective bundle $\pp_X(V)$ we see that the map $s:A_1(X)\to A_1(\pp_X(V))$ induced by $s_X$ splits the sequence above. This fixes $\varphi$. 
 \begin{definition}
Let $\beta\in A_1(X)$, then we call $\tilde{\beta}:=s(\beta)$ the lifting of $\beta$.
\end{definition}
In these notations any class of a curve in $\pp_X(V)$ can be written uniquely as $\tilde{\beta}+qf_X$, for some $\beta\in A_1(X)$ and some $q\in\Z$. Let us consider $\alpha\in A^k(\overline{M}_{0,n}(\pp_X(V),\tilde{\beta}+qf_X))$ such that
\begin{equation}\label{dimension}
\vdim\overline{M}_{0,n}(\pp_X(V),\tilde{\beta}+qf_X)-k=\vdim \overline{M}_{0,n}(X,\beta).
\end{equation}
We will say that $(\beta, W, q, \alpha)$ satisfies condition (\ref{dimension}) or when there is no risk of confusion that $(W, q, \alpha)$ satisfies condition (\ref{dimension}).
\\In the same way as in the case of blow-ups, we will relate genus-zero GW invariants of  $\pp_X(V)$ to genus-zero GW invariants of $X$.
\begin{remark}\label{smirr}
Let $X$ be a convex variety, $\beta\in A_1(X)$, $W$ a rank-$(r-1)$ vector bundle on $X$, $q\in\Z$ and $\alpha\in A^*(\overline{M}_{0,n}(\pp_X(V),\tilde{\beta}+qf_X))$ satisfying condition (\ref{dimension}). Let $Z_1,...Z_k$ be the connected components of $\overline{M}_{0,n}(X,\beta)$. As $X$ is convex, by dimensional reasons $$(\bar{p}_X)_* \left(\alpha\cdot[\overline{M}_{0,n}(\pp_X(V),\tilde{\beta}+qf_X)]^{\vv}\right)=\sum_{i=1}^k N_i[Z_i],$$ for some $N_i\in\Q$, possibly zero.
\\In particular, if $X=\pp^1$, then $\overline{M}_{0,n}(X,\beta)$ is smooth, irreducible (and unobstructed) and therefore $$(\bar{p}_X)_*\left(\alpha\cdot[\overline{M}_{0,n}(\pp_X(V),\tilde{\beta}+qf_X)]^{\vv}\right)=N[\overline{M}_{0,n}(X,\beta)]^{\vv}$$ for some $N\in\Q$. 
\end{remark}
\begin{definition}\label{locconst}(i) In notations as above, we consider the locally constant function
\begin{equation*}
N_{X,W,\beta,q}(\alpha):\overline{M}_{0,n}(X,\beta)\to\Q
\end{equation*}
defined  by $N_{X,W,\beta,q}(\alpha)=N_i$ on $Z_i$.
\\(ii) Let $X=\pp^1$, $d=(d_1,...,d_{r-1})\in\Z^{r-1}$, $V=\mathcal{O}\oplus \mathcal{O}_{\pp^1}(d_1)\oplus...\oplus\mathcal{O}_{\pp^1}(d_{r-1})$,  and $k=(k_1,...,k_n)\in \N^n$. Let $\xi_{\pp_X(V)}=c_1(\mathcal{O}_{\pp_X(V)}(1))$ and assume that $\alpha=ev^*_i\xi^{k_i}$ satisfies the dimension condition (\ref{dimension}). We define $N(q,d,k)\in \Q$ by the formula $$(\bar{p}_X)_*\left(\alpha\cdot[\overline{M}_{0,n}(\pp_X(V),\tilde{1}+qf_X)]^{\vv}\right)=N(q,d, k)[\overline{M}_{0,n}(X,1)].$$
\end{definition}
\begin{remark}Let $X$ and $Y$ be smooth projective varieties, let $f:Y\to X$ be a morphism and $\beta_Y\in A_1(Y)$ such that $f_*\beta_Y=\beta$. Let $W$ be a vector bundle on $X$. Then there exists an induced map $h:\pp_Y(f^*V)\to\pp_X(V)$. This induces a map $\bar{h}:\overline{M}_{0,n}(\pp_Y(f^*V),\tilde{\beta}_Y+qf_Y)\to \overline{M}_{0,n}(\pp_X(V),\tilde{\beta}+qf_X)$
\end{remark}
\begin{proposition}\label{projb} Let $X$ be smooth convex projective varieties and let $f:Y\to X$ be a morphism of smooth projective varieties such that $$f_*:A_1(Y)\to A_1(X)$$ is injective. Let $W$ be a vector bundle on $X$ and $$\alpha\in A^*(\overline{M}_{0,n}(\pp_X(V),\tilde{\beta}+qf_X))$$ which satisfies the dimension condition (\ref{dimension}). Then
\\(i) $(\beta_Y, f^*W,q, \bar{h}^*\alpha)$ satisfies condition \ref{dimension}.
\\(ii)Let $Z_1,...,Z_k$ be the connected components of $\overline{M}_{0,n}(X,\beta)$. For any $i\in\{1,...,k\}$ let $N_{X,W,\beta,q}(\alpha)|_{Z_i}$ be the number from definition \ref{locconst} which corresponds to $Z_i$. Then, we have an equality
\begin{equation*}
%N_{Y,f^*W,\beta_Y,q}(\bar{h}^*\alpha)|_{V_i}=N_{X,W,\beta,q}(\alpha)|_{Z_i},\ \forall i\in\{1,...,k\}.
(\bar{p}_Y)_*\left((\bar{h}^*\alpha)\cdot[\overline{M}_{0,n}(\pp_{Y}(f^*V),\tilde{\beta_Y}+qf_Y)]^{\vv}\right)=\sum_{i=1}^kN_{X,W,\beta,q}(\alpha)|_{Z_i}[Z_i]
\end{equation*}
\end{proposition}
\begin{proof}
(i) Let us consider the following cartesian diagram
\begin{equation*}
 \xymatrix{\overline{M}_{0,n}(\pp_{Y}(f^*V),\tilde {\beta_Y}+qf_Y)\ar[d]_{\bar{p}_{Y}}\ar[r]^{\bar{h}}&\overline{M}_{0,n}(\pp_X(V),\tilde{\beta}+qf_X)\ar[d]^{\bar{p}_X}\\
\overline{M}_{0,n}(Y,\beta_Y)\ar[r]^{\bar{f}}&\overline{M}_{0,n}(X,\beta).}
\end{equation*}
Applying Proposition \ref{cartesian} we obtain that the relative obstruction theories $E_{\bar{f}}$ and $E_{\bar{h}}$ are compatible. This shows that $(\beta_Y, f^*W,q, \bar{h}^*\alpha)$ satisfies condition \ref{dimension}.
\\(ii)Without loss of generality we may assume that $\overline{M}_{0,n}(X,\beta)$ is irreducible. Let us denote its unique connected component by $Z_1$. By Definition \ref{locconst} we have that $$(\bar{p}_X)_* \left(\alpha\cdot[\overline{M}_{0,n}(\pp_X(V),\tilde{\beta}+qf_X)]^{\vv}\right)=N_{X,W,\beta,q}(\alpha)|_{Z_1}[Z_1].$$ Applying Proposition \ref{relations} (iii) to the above diagram we obtain that
\begin{equation}\label{commut}
 \bar{f}^!((\bar{p}_X)_*\alpha\cdot[\overline{M}_{0,n}(\pp_X(V),\tilde{\beta}+qf_X)]^{\vv})=(\bar{p}_{Y})_*\bar{f}^!(\alpha\cdot[\overline{M}_{0,n}(\pp_X(V),\beta+qf_X)]^{\vv}).
\end{equation}
\\As the obstruction theories $\bar{p}_Y^*E_{\bar{f}}$ and $E_{\bar{h}}$ are compatible by Proposition \ref{cartesian}, we obtain that \begin{equation}\label{iii}\bar{f}^!_{\mathfrak{E}_{\bar{f}}}=\bar{h}^!_{\mathfrak{E}_{\bar{h}}}.
\end{equation}
 This shows that $$\bar{f}^!\left(\alpha\cdot[\overline{M}_{0,n}(\pp_X(V),\tilde{\beta}+qf_X)]^{\vv}\right)=\bar{h}^*\alpha\cdot \bar{f}^![\overline{M}_{0,n}(\pp_X(V),\tilde{\beta}+qf_X)]^{\vv}.$$ By Corollary \ref{virtual} we obtain $\bar{f}^![\overline{M}_{0,n}(X,\beta)]^{\vv}=[\overline{M}_{0,n}(Y,\beta_Y)]^{\vv}$ and using relation (\ref{iii}) we get
\begin{equation}\label{sus}
 \bar{f}^![\overline{M}_{0,n}(\pp_X(V),\tilde{\beta}+qf_X)]^{\vv}=[\overline{M}_{0,n}(\pp_{Y}(f^*V),\beta_Y+qf_Y)]^{\vv}.
\end{equation}
From equations (\ref{commut}) and (\ref{sus}) we see that
\begin{equation}\label{nj}
N_{X,W,\beta,q}(\alpha)|_{Z_1}=(\bar{p}_Y)_*\left((\bar{h}^*\alpha)\cdot[\overline{M}_{0,n}(\pp_{Y}(f^*V),\tilde{\beta_Y}+qf_Y)]^{\vv}\right).
\end{equation}
This concludes the proof.
\end{proof}
\begin{corollary}\label{convexbdl} We follow the notations of Proposition \ref{projb}. Let $X$ be a homogeneous space, $Y$ a rational curve. Let $$\alpha:=\prod_{i=1}^nev_i^*\xi^{k_i}\in A^*(\overline{M}_{0,n}(\pp_{X}(V),\tilde{\beta}+qf_{X}))$$ satisfy condition (\ref{dimension}). Suppose $f^*V\simeq \mathcal{O}\oplus\mathcal{O}(d_1)\oplus...\oplus\mathcal{O}(d_{r-1})$. Then $$(\bar{p}_{X})_* \left(\alpha\cdot[\overline{M}_{0,n}(\pp_{X}(V),\tilde{\beta}+qf_{X})]^{\vv}\right)=N(q,d, k)[\overline{M}_{0,n}(Y,\beta)]^{\vv}.$$
\end{corollary}
\begin{proof}
Since $X$ is a homogeneous space we have that  $\overline{M}_{0,n}(X,\beta)$ is irreducible. Let $h:\pp_C(f^*V)\to \pp_X(V)$ be the map induced by $f$. The claim follows from Proposition \ref{projb} with $Y\simeq\pp^1$, $\beta_Y=1$ and the equality $h^*\xi_{\pp_X(V)}=\xi_{\pp_Y(f^*V)}$.
\end{proof}
\begin{corollary}\label{jump} Let $d=(d_1,...,d_{r-1})$, $e=(e_1,...,e_{r-1})$ with $d_i,\ e_i\geq 0$, $\forall\ 0\leq i\leq r$ and $\sum_{i=1}^{r-1} d_i=\sum_{i=1}^{r-1} e_i$. Then $N(q,d,k)=N(q,e,k)$.
\end{corollary}
\begin{proof} In notations as in \ref{projb}, let $d:=\sum d_i$. By  \cite{ik} Lemma 2.4 there exists $m\gg n$  such that there exist $Y_d$, $Y_e$ rational curves on $G(r-1,m)$ with embeddings $i_d:Y_d\to G(r,m)$, $i_e:Y_e\to G(r,m)$ such that $i_d^*S\simeq\oplus\mathcal{O}(d_1)\oplus...\oplus\mathcal{O}(d_r)$ and  $i_e^*S\simeq \mathcal{O}\oplus\mathcal{O}(e_1)\oplus...\oplus\mathcal{O}(e_r)$. Here $S$ is the dual of the tautological subbundle on $G(r,m)$. Let $X=G(r,m)$, $V=S\oplus\mathcal{O}$ and $\beta=d$ and let $h_d:\pp_{Y_d}(i_d^*V)\to \pp_X(V)$ and $h_e:\pp_{Y_e}(i_e^*V)\to \pp_X(V)$ be the maps induced by $i_d$, $i_e$. Let us apply Corollary \ref{convexbdl} to $i_d$. From the fact that $\overline{M}_{0,n}(G(r,m),d)$ is irreducible and the equality $h_d^*\xi_{\pp_X(V)}=\xi_{\pp_C(i_d^*V)}$ we obtain $$(\bar{p}_X)_*\left((\bar{h}^*\alpha)\cdot[\overline{M}_{0,n}(\pp_{X}(V),\tilde{\beta_X}+qf_X)]^{\vv}\right)=N(q,d,k).$$ Similarly, by applying Corollary \ref{convexbdl} to $i_e$ we obtain $$(\bar{p}_X)_*\left((\bar{h}^*\alpha)\cdot[\overline{M}_{0,n}(\pp_{X}(V),\tilde{\beta_X}+qf_X)]^{\vv}\right)=N(q,d,k).$$ Comparing the two equalities we obtain that $N(q,d,k)=N(q,e,k)$.
\end{proof}
%\begin{corollary} Since $C$ is isomorphic to $\pp^1$, $f^*V\simeq\oplus_{i=1}^r\mathcal{O}(d_i)$. By the projection formula $h_*\sum_ic_1(\mathcal{O}(d_i))=\beta\cdot c_1(V)$. It can be easily seen that $h^*\xi_{\pp_X(V)}=\xi_{\pp_C(f^*V)}$. Since the numbers $N(q,d, k)$ depend only on $\sum_id$ and not on the decomposition $\oplus_{i=1}^r\mathcal{O}(d_i)$, the conclusion follows.
%\end{corollary}
Let us now extend the result to a more general base $X$.
\begin{setting}\label{notatii}We consider a homogeneous space space $\pp$ and $g:X\to\pp$ be a closed embedding of a smooth projective variety $X$ in $\pp$. Let $V$ be a vector bundle on $X$ such that there exists a vector bundle $W\oplus \C_{\pp}$ on $\pp$ with $V=g^*(W\oplus \C_{\pp})$ and let $p:\pp_X(V)\to X$ be the associated projective bundle. In notations as above we have an induced map $\bar{p}_X:\overline{M}_{0,n}(\pp_{X}(V),\tilde{\beta}+qf_X)\to\overline{M}_{0,n}(X,\beta)$. \end{setting}
\begin{definition} Let $V$ be a vector bundle on a smooth projective variety $X$ and let $i:\pp^1\to X$ be a closed embedding of a projective line in $X$. Let $i^*V\simeq\mathcal{O}(d_1)\oplus...\oplus\mathcal{O}(d_{r-1})$ for some $d_1,...,d_{r-1}\in\Z$. We say that $i^*V$ is positive if $d_i\geq0$, $\forall\ 0\leq i\leq r-1$.
\end{definition}

\begin{corollary}\label{genbdl} In notations as in \ref{notatii}, let $$\alpha:=\prod_{i=1}^nev_i^*\xi^{k_i}\in A^*(\overline{M}_{0,n}(\pp_X(V),\tilde{\beta}+qf_X))$$ satisfy condition (\ref{dimension}). Suppose moreover that the restriction of $V$ to any curve of class $\beta$ is positive. Then we have that $$(\bar{p}_X)_* \left(\alpha\cdot[\overline{M}_{0,n}(\pp_X(V),\tilde{\beta}+qf_X)]^{\vv}\right)=N(q,d, k)[\overline{M}_{0,n}(X,\beta)]^{\vv},$$
for some $d=(d_1,...,d_n)$ which satisfies $\sum_id_i=\beta\cdot c_1(V)$.
\end{corollary}
\begin{proof}
 The proof follows from Proposition \ref{projb}, Corollary \ref{convexbdl} and Corollary \ref{jump}.
\end{proof}

\begin{remark}
 In Corollary \ref{genbdl} we have shown that we can compute GW invariants of projective bundles in terms of GW invariants of the base $X$ and GW invariants of a projective bundle over $\pp^1$. The latter can be analyzed using toric methods (see \cite{e1}, \cite{e2}). More precisely, we can compute in this way GW-invariants of $\pp_X(V)$ with at least $\vdim \overline{M}_{0,n}(X,\beta)$ insertions that are pull-backs from $X$.
\end{remark}

\subsection{Costello's push-forward formula}
We can use the basic properties of virtual pull-backs (push-forward and functoriality) to give a short proof of a particular case of Costello's push-forward formula in \cite{c}. We recall the set-up from \cite{c}.
\\Let us consider a cartesian diagram
\begin{equation*}
 \xymatrix{F\ar[d]_{p_1}\ar[r]^f&G\ar[d]^{p_2}\\
\mathfrak{M}_1\ar[r]^g&\mathfrak{M}_2}
\end{equation*}
such that
\begin{enumerate}
\item $f$ is a proper morphism;
\item $\mathfrak{M}_1$ and $\mathfrak{M}_2$ are Artin stacks of the same pure dimension;
\item $g$ is a DM-type morphism of degree $d$;
\item $F$ and $G$ are DM-stacks equipped with perfect relative obstruction theories $E_{F/\mathfrak{M}_1}$ and $E_{G/\mathfrak{M}_2}$ inducing virtual classes $[F]^{\vv}$ and $[G]^{\vv}$;
\item $E_{F/\mathfrak{M}_1}\simeq f^*E_{G/\mathfrak{M}_2}$.
\end{enumerate}
\begin{proposition}\label{costello} Under the assumptions above, if moreover $g$ is projective, then $f_*[F]^{\vv}=d[G]^{\vv}$.
 \end{proposition}
%\begin{proposition} Under the assumptions above if $g$ is projective, then $f_*[F]^{\vv}=d[G]^{\vv}$.
 %\end{proposition}
\begin{proof}
 As $E_{F/\mathfrak{M}_1}$ and $E_{G/\mathfrak{M}_2}$ are perfect, $p_1$ and $p_2$ induce pull-back morphisms and $E_{F/\mathfrak{M}_1}=f^*E_{G/\mathfrak{M}_2}$ implies $p_1^!$ is induced by $p_2^!$. Applying Theorem \ref{relations} (i) we get $f_*p_1^![\mathfrak{M}_1]=p_2^!g_*[\mathfrak{M}_1]$. Using the fact that $g_*[\mathfrak{M}_1]=d[\mathfrak{M}_2]$ and the definition of virtual classes we get
\begin{equation*}
 f_*[F]^{\vv}=d[G]^{\vv}.
\end{equation*}
\end{proof}
\begin{remark} In \cite{c}, Theorem 5.0.1 $g$ is not assumed to be projective. We impose this condition in order to be able to push-forward along $g$ and apply Theorem \ref{relations}. If $\mathfrak{M}_1$ and $\mathfrak{M}_2$ are DM-stacks then it is enough to assume $g$ is proper. The proof of Proposition \ref{costello} applies unchanged to this case.
\\More generally, Theorem 5.0.1 in \cite{c} follows from the proof of Theorem \ref{relations} (i) which is very similar to Costello's proof.
\end{remark}

\end{document}